\documentclass[11pt]{scrartcl}

\usepackage[utf8]{inputenc}
\usepackage{amsmath}
\usepackage{amsthm} 
\usepackage{thmtools}
\usepackage{thm-restate}

\usepackage[inline]{enumitem}

\usepackage{algorithm}
\usepackage[noend]{algpseudocode}
\usepackage{hyperref}
\usepackage[nameinlink]{cleveref} 
\usepackage{config}
\usepackage{titling}
\usepackage{macros}
\usepackage{tikzmacros}
\usepackage{textpos}
\usepackage{soul}

\pgfdeclarelayer{background}
\pgfdeclarelayer{foreground}
\pgfsetlayers{background,main,foreground}

\pgfkeys{%
	/tikz/on layer/.code={
		\pgfonlayer{#1}\begingroup
		\aftergroup\endpgfonlayer
		\aftergroup\endgroup
	},
	/tikz/node on layer/.code={
		\pgfonlayer{#1}\begingroup
		\expandafter\def\expandafter\tikz@node@finish\expandafter{\expandafter\endgroup\expandafter\endpgfonlayer\tikz@node@finish}%
	}
}

\setlist[itemize]{topsep=0pt,partopsep=0pt,itemsep=0pt,parsep=0pt}
\setlist[itemize,1]{label={\small\textbullet}}
\setlist[itemize,2]{label={\tiny\textbullet}}
\setlist[itemize,3]{label=$\cdot$}
\setlist[enumerate]{topsep=0pt,partopsep=0pt,itemsep=0pt,parsep=0pt}
\setlist[enumerate,1]{label=\roman*)}
\setlist[enumerate,2]{label=\alph*)}
\setlist[enumerate,3]{label=\arabic*)}

\hypersetup{
	colorlinks=true,
	linkcolor=AO!70!black,
	citecolor=AO!70!black,
	urlcolor=AppleGreen!70!black,
	bookmarksopen=true,
	bookmarksnumbered,
	bookmarksopenlevel=2,
	bookmarksdepth=3
}

\newif\ifcomment
\commentfalse
\commenttrue

\title{Braces of perfect matching width two\thanks{This work was mainly conducted while all three authors were are the Technical University Berlin and supported by the European Research Council
		(ERC) under the European Union’s Horizon 2020 research and innovation programme (ERC Consolidator
		Grant DISTRUCT, grant agreement No 648527).}}
\predate{}
\date{}
\postdate{}

\preauthor{}
\DeclareRobustCommand{\authorthing}{
	\begin{center}
	    Archontia C.\@ Giannopoulou -- Department of Informatics and Telecommunications, National and Kapodistrian University of Athens\\
	    \href{mailto:archontia.giannopoulou@gmail.com}{archontia.giannopoulou@gmail.com}\\
	    Meike Hatzel\thanks{Meike Hatzel's research was supported by the Federal Ministry of Education and Research (BMBF) and by a fellowship within the IFI programme of the German Academic Exchange Service (DAAD).} -- National Institute of Informatics, Tokyo\\
	    \href{mailto:research@meikehatzel.com}{research@meikehatzel.com}\\
	    Sebastian Wiederrecht\thanks{Supported by the ANR project ESIGMA (ANR-17-CE23-0010) and the Institute for Basic Science (IBS-R029-C1).} -- Institute for Basic Science, Daejeon, South Korea.\\
	    \href{mailto:sebastian.wiederrecht@gmail.com}{sebastian.wiederrecht@gmail.com}
\end{center}}
\author{\authorthing}
\postauthor{}

\begin{document}
\maketitle
\begin{textblock}{20}(-1.9,5.5)
   \includegraphics[width=80px]{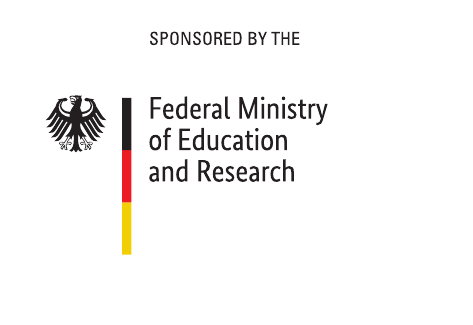}%
\end{textblock}

\begin{abstract}
Perfect matching width is a treewidth-like parameter designed for graphs with perfect matchings.
The concept was originally introduced by Norine for the study of non-bipartite Pfaffian graphs.
Additionally, perfect matching width appears to be a useful structural tool for investigating matching minors, a specialised version of minors related to perfect matchings.
In this paper we lay the groundwork for understanding the interaction of perfect matching width and matching minors by establishing tight connections between the perfect matching width of any matching covered graph $G$ and the perfect matching width of its bricks and braces (a matching theoretic version of blocks) and proving that perfect matching width is almost monotone under the matching minor relation.
As an application, we give several characterisations for braces of perfect matching width two, including one that allows for a polynomial time recognition algorithm.

\end{abstract}

	

\section{Introduction}

The significance of matching minors, especially for the study of matching theoretic properties of bipartite graphs, was first observed by Little\footnote{Although the word `matching minor' did not appear before the work of Norine and Thomas \cite{norine2005matching,norine2007generating}.} \cite{little1975characterization}, who used bisubdivisions to characterise Pfaffian bipartite graphs.
In the pursuit of a description of non-bipartite Pfaffian graphs, Norine introduced a matching theoretic analogue of treewidth, called \emph{perfect matching width}~\cite{norine2005matching}, and gave an algorithm that decides
whether a graph from a class of bounded perfect matching width is Pfaffian in XP-time.
While the general problem of characterising non-bipartite Pfaffian graphs remains open, perfect matching width appears to be a valuable tool for the study of matching minors.
For further information on Pfaffian graphs the reader is referred to the surveys of McCuaig and Thomas~\cite{mccuaig2004polya,thomas2006survey}.

Bipartite and non-bipartite graphs with perfect matchings behave in fundamentally different ways.
One particular point of contrast is the matching minor relation.
There exists an infinite anti-chain of bricks~\cite{norine2008minimally}, however, for braces, no infinite anti-chain of braces is known for the matching minor relation.
This hints at the existence of some more fundamental challenges regarding non-bipartite graphs that have to be overcome for establishing a theory of matching minors.
Hence in this paper we focus on bipartite graphs and braces in particular.
We mostly consider the class of \emph{matching covered} graphs, that is, graphs in which every edge lies in some perfect matching.
Please note that, since matching covered graphs encode all strongly connected digraphs (see for example~\cite{robertson1999permanents,hatzel2019cyclewidth}), matching covered graphs are already a significant and non-trivial class.

In the following we aim at denoting bipartite graphs with  a capital `$B$' instead of a `$G$' to highlight the parts where we specialise on bipartite graphs.
Moreover, we assume every bipartite graph $B$ to come with a two-partition into the \emph{colour classes} $V_1$ and $V_2$, where vertices of $V_1$ are usually depicted as filled vertices and those of $V_2$ are empty in figures.

We give a short overview of the main results of this paper.
In \cref{sec:basicproperties} we consider the following two questions:
How does the perfect matching width of a graph relate to the perfect matching width of its matching minors? And, can we determine the perfect matching width of a graph from the perfect matching width of its bricks and braces (its uniquely determined building blocks~\cite{lovasz1987matching})?
In particular we prove the following two theorems.

\begin{restatable}{theorem}{pmwmatminor}\label{thm:pmwandmatminors}
    Let $G$ be a matching covered graph and $H$ be a matching minor of $G$.
    Then $\pmw{H}\leq 2\pmw{G}$.
\end{restatable}

\begin{restatable}{theorem}{pmwbricksbraces}\label{thm:pmwandbricksandbraces}
Let $G$ be a matching covered graph and $H$ be a brick or brace of $G$ maximising $\pmw{H}$ over all bricks and braces of $G$.
Then $\frac{1}{2}\pmw{H}\leq\pmw{G}\leq\pmw{H}$.
\end{restatable}

Next we consider the graphs of small perfect matching width.
Having established that the maximum perfect matching width among the braces yields an upper bound on the perfect matching width of the graph itself, we start by characterising the braces of width two in \cref{sec:pmw2}.
We provide two different characterisations, one in terms of conformal subgraphs of a well-defined class of edge-maximal braces of perfect matching width two, which we refer to as \emph{bipartite ladders}, and the other in form of specific elimination orderings.

Both characterisations allow us to obtain further results.
In \cref{sec:comparissontotreewidth} we make use of the bipartite ladders to show that, while the perfect matching width of every graph is bounded from above by its treewidth, for every $k\in\N$ there exists a brace of perfect matching width two whose treewidth is at least $k$.
The characterisation via elimination orderings yields a polynomial time algorithm for the construction of a perfect matching decomposition of width two if one exists, which we prove in \cref{sec:algorithm}.

The factor two in \cref{thm:pmwandmatminors} stops us from directly lifting the results for braces to all bipartite graphs of perfect matching width two.
However, in \cref{sec:Mpmw}, we consider a more restrictive version of perfect matching width, the \emph{$M$-perfect matching width}, denoted $\Mpmw{M}{},$ which is known to be parametrically equivalent to perfect matching width~\cite{hatzel2019cyclewidth}.
For $M$-perfect matching width we are able to prove a stronger version of \cref{thm:pmwandmatminors} removing the factor two.

This allows us to give a full characterisation of all bipartite graphs $B$ that have a perfect matching $M$ such that $\Mpmw{M}{B}=2$.

\begin{figure}[!ht]
	\centering
	\begin{tikzpicture}[scale=0.9]
		\pgfdeclarelayer{background}
		\pgfdeclarelayer{foreground}
		\pgfsetlayers{background,main,foreground}
		
		\node[v:ghost] (mid) {};
		
		\node (k5) [v:ghost,position=180:25mm from mid] {};   
		\node (k33) [v:ghost,position=0:25mm from mid] {};
		
		\node (v1) [v:main,position=135:8mm from k5] {}; 
		\node (v2) [v:mainempty,position=225:8mm from k5] {};
		\node (v3) [v:main,position=315:8mm from k5] {};
		\node (v4) [v:mainempty,position=45:8mm from k5] {};
		\node (v5) [v:mainempty,position=135:16mm from k5] {};
		\node (v6) [v:main,position=225:16mm from k5] {};
		\node (v7) [v:mainempty,position=315:16mm from k5] {};
		\node (v8) [v:main,position=45:16mm from k5] {};
		
		\node (k5label) [v:ghost,position=270:21mm from k5] {cube};
		
		\node (u1) [v:main,position=126:8mm from k33] {};
		\node (u2) [v:mainempty,position=198:8mm from k33] {};
		\node (u3) [v:main,position=270:8mm from k33] {};
		\node (u4) [v:mainempty,position=342:8mm from k33] {};
		\node (u5) [v:main,position=54:8mm from k33] {};
		\node (u6) [v:mainempty,position=126:16mm from k33] {};
		\node (u7) [v:main,position=198:16mm from k33] {};
		\node (u8) [v:mainempty,position=270:16mm from k33] {};
		\node (u9) [v:main,position=342:16mm from k33] {};
		\node (u10) [v:mainempty,position=54:16mm from k33] {};

		\node (k33label) [v:ghost,position=270:21mm from k33] {$\mathscr{M}_{10}$};
		
		\begin{pgfonlayer}{background}
			
			\draw[e:main] (v1) to (v2);
			\draw[e:main] (v2) to (v3);
			\draw[e:main] (v3) to (v4);
			\draw[e:main] (v4) to (v1);
			
			\draw[e:main] (v5) to (v6);
			\draw[e:main] (v6) to (v7);
			\draw[e:main] (v7) to (v8);
			\draw[e:main] (v8) to (v5);
			
			\draw[e:main] (v1) to (v5);
			\draw[e:main] (v2) to (v6);
			\draw[e:main] (v3) to (v7);
			\draw[e:main] (v4) to (v8);
			
			\draw[e:main] (u1) to (u2);
			\draw[e:main] (u2) to (u3);
			\draw[e:main] (u3) to (u4);
			\draw[e:main] (u4) to (u5);
			\draw[e:main] (u5) to (u6);
			\draw[e:main] (u6) to (u7);
			\draw[e:main] (u7) to (u8);
			\draw[e:main] (u8) to (u9);
			\draw[e:main] (u9) to (u10);
			\draw[e:main] (u10) to (u1);
			
			\draw[e:main] (u1) to (u6);
			\draw[e:main] (u2) to (u7);
			\draw[e:main] (u3) to (u8);
			\draw[e:main] (u4) to (u9);
			\draw[e:main] (u5) to (u10);
			
		\end{pgfonlayer}
	\end{tikzpicture}
	\caption{The braces \emph{cube} and the Möbius ladder $\mathscr{M}_{10}$ of order 10.}
	\label{fig:cubeandM10}
\end{figure}

\begin{restatable}{theorem}{cubeormoebiustheorem}
    \label{thm:Mpmw2}
    Let $B$ be a bipartite graph with a perfect matching.
    The following statements are equivalent.
    \begin{enumerate}
	   \item There exists a perfect matching $M\in\Perf{B}$ such that $\Mpmw{M}{B}=2$,
	   \item for all $M\in\Perf{B}$ we have $\Mpmw{M}{B}=2$, and
	   \item $B$ does not contain the cube or $\mathscr{M}_{10}$ as a matching minor (see \cref{fig:cubeandM10}),
    \end{enumerate}
\end{restatable}

\subsection{Preliminaries}

All graphs considered in this article are finite and simple, that is, we do not allow for loops or parallel edges.
Let $G$ be a graph.
A set $M\subseteq\E{G}$ of edges is a \emph{matching} if no two edges in $M$ share an endpoint.
A matching $M$ is called \emph{perfect} if every vertex of $G$ is an endpoint of some edge of $M$.
We denote the set of perfect matchings of $G$ by $\Perf{G}$.
If $G$ is connected and every edge of $G$ is contained in some perfect matching of $G$ we say that $G$ is \emph{matching covered}.
The aim of Matching Theory is to study the structural properties of graphs with perfect matchings, and, within its context a plethora of results revealing a rich structural theory has appeared.
For an in-depth exposition of Matching Theory the interested reader is referred to the book by Lov\'asz and Plummer \cite{lovasz2009matching}.

For the study of perfect matchings, there are several adaptions of standard terminology from graph theory to make sure the existence of a perfect matching is preserved.
Let $G$ be a graph with a perfect matching $M$.
A set $X\subseteq\V{G}$ is \emph{conformal} if $G-X$ has a perfect matching, it is said to be \emph{$M$-conformal} if $M$ contains a perfect matching of $G-X$.
Similarly a subgraph $H\subseteq G$ is \emph{conformal} if $\V{H}$ is, and it is \emph{$M$-conformal} if $M$ contains perfect matchings of $G-\V{H}$ and of $H$.
A \emph{bisubdivision} of $G$ is a graph $G'$ obtained by replacing every edge of $G$ with a path of odd length (possibly one) whose internal vertices are fresh vertices.
A \emph{bicontraction} is the operation of contracting both edges incident with a vertex of degree two at the same time.
Finally, a \emph{matching minor} is a graph $H$ that can be obtained by a series of bicontractions from a conformal subgraph of $G$.
Note that, if $G'$ is a bisubdivision of $G$, then $G$ is indeed a matching minor of $G'$.

A subset of vertices $S$ in a graph $G$ is a \emph{separator}, if $G-S$ is not connected and $G$ is \emph{$k$-connected} if $G$ does not contain a separator of size $k.$
Let $G$ be a graph and $X\subseteq\V{G}$.
The \emph{edge cut} with \emph{shores} $X$ and $\Compl{X}\coloneqq\V{G}\setminus X$ is the set $\CutG{G}{X}\coloneqq\CondSet{e\in\E{G}}{e\cap X\neq\emptyset,~e\cap\Compl{X}\neq\emptyset}$.
The \emph{matching porosity} of $\CutG{G}{X}$ is the value
\begin{equation*}
	\MatPor{\CutG{G}{X}} \coloneqq \max_{M\in\Perf{G}}\Abs{M\cap\CutG{G}{X}}.
\end{equation*}
We say that a tree $T$ is \emph{cubic} if every non-leaf vertex of $T$ is of degree three.
By $\Leaves{T}$ we denote the set of leaves of any tree, and if $t_1t_2\in\E{T}$ is an edge of $T$, we denote by $T_{t_i}$ the unique component of $T-t_1t_2$ which contains $t_i$ for any $i\in\Set{1,2}$.
Let $T$ be a subcubic tree.
We can obtain a cubic tree $T'$ from $T$ by iteratively choosing a degree $2$ vertex and contracting one of its two incident edges.
The tree $T'$ is, up to isomorphism, uniquely determined by $T$ and we call $T'$ the tree obtained from $T$ by \emph{trimming}.
Note that $\Leaves{T}=\Leaves{T'}$.

A \emph{perfect matching decomposition} of $G$ is a tuple $\Brace{T,\DecompBijection{}}$ where $T$ is a cubic tree and $\DecompBijection{} \colon \Leaves{T}\rightarrow\V{G}$ is a bijection.
We associate with every edge $t_1t_2\in\E{T}$ a bipartition of $\V{G}$ into the sets $\DecompBijection{T_{t_i}}\coloneqq \bigcup_{\ell\in\Leaves{T_{t_i}}}\Set{\DecompBijection{\ell}}$ for both $i\in\Set{1,2}$.
Moreover, we can associate the edge cut $\CutG{G}{\DecompBijection{T_{t_1}}}=\CutG{G}{\DecompBijection{T_{t_2}}}$ of $G$ with $t_1t_2$.
As a shorthand, we denote this edge cut by $\CutG{G}{t_1t_2}$.
The \emph{width} of $\Brace{T,\DecompBijection{}}$ is defined as $\max_{e\in\E{T}}\MatPor{\CutG{G}{e}}$, and the \emph{perfect matching width} of $G$, denoted by $\pmw{G}$, is the minimum width over all perfect matching decompositions of $G$.

Let $M$ be a perfect matching of a graph $G$.
An \emph{$M$-decomposition} of $G$ is a perfect matching decomposition $\Brace{T,\DecompBijection{}}$ of $G$ such that for every edge $e=uv \in M$, the leaves $t_u$ and $t_v$ of $T$ with $\DecompBijection{t_u} = u$ and $\DecompBijection{t_v} = v$ are
both adjacent to a vertex $t_{uv}$ of $T$.
Note that a perfect matching decomposition $\Brace{T,\DecompBijection{}}$ of $G$ is an $M$-decomposition if and only if for every $t_1t_2\in\E{T-\Leaves{T}}$ the sets $\DecompBijection{T_{t_i}}$ are $M$-conformal.
The \emph{$M$-perfect matching width} of $G$, denoted by $\Mpmw{M}{G}$, is now the minimum width over all $M$-decompositions of $G$.
Perfect matching width and $M$-perfect matching width are parametrically equivalent.

\begin{restatable}[\cite{hatzel2019cyclewidth}\protect\footnote{Note that the statement slightly differs from the one in the original paper, it still follows from the proof provided there.}]{theorem}{mpmwandpmw}\label{thm:Mpmwandpmw}
Let $G$ be a graph with a perfect matching $M$.
Then $\frac{1}{2}\Mpmw{M}{G}\leq\pmw{G}\leq\Mpmw{M}{G}$.
\end{restatable}

\subsubsection*{Tight cuts, bricks and braces}

Similar to (undirected) treewidth and directed treewidth being used to study minors and butterfly minors, perfect matching width can be considered a tool for the study of matching minors.

Let $X\subseteq\V{G}$ be a set of vertices.
The edge cut $\CutG{G}{X}$ is called \emph{tight} if $\MatPor{\CutG{G}{X}}=1$, it is called \emph{trivial} if $\Abs{X}=1$ or $\Abs{\Compl{X}}=1$, we say that $X$ \emph{induces} a tight cut, if $\CutG{G}{X}$ is tight.
Note that the shores of a tight cut are always odd.
If $X$ is a shore of a tight cut, we call the operation of identifying $X$ into a single vertex $v_{X}$ and deleting all resulting loops and parallel edges a \emph{tight cut contraction}.
Observe that tight cut contractions of a matching covered graph are again matching covered.
Let $u$ be the unique vertex of $\Compl{X}$ which is incident with the edge of $M$ in $\CutG{G}{X}$, then we call the perfect matching \hypertarget{remainderDef}{$\Remainder{M}{G_X} \coloneqq \Brace{M\cap \E{\InducedSubgraph{G}{\Compl{X}}}}\cup\Set{uv_X}$} the \emph{residual} of $M$ in $G_X$.

A matching covered graph without a non-trivial tight cut is called a \emph{brace} if it is bipartite and a \emph{brick} otherwise.
So every matching covered graph either is a brace, a brick or has a \nontrivial tight cut and therefore can be decomposed into two smaller matching covered graphs.
One can continue with this process of decomposing along tight cuts in the two smaller graphs until there are no more \nontrivial tight cuts to be found.
This process yields a list of bricks and braces and is known as the \emph{tight cut decomposition} of $G$.
A famous result by Lov{\'a}sz states that this list depends only on $G$ and does not depend on the choices of the tight cuts, thus it is unique for every fixed graph $G$.
	
\begin{theorem}[\cite{lovasz1987matching}]\label{thm:lovasztightcuts}
	Any two tight cut decomposition procedures of a matching covered graph $G$ yield the same list of bricks and braces.
\end{theorem}

We say that two cuts $\Cut{S},\Cut{T}$ \emph{cross} if all four of the following sets $S\cap T,$ $S\cap \Compl{T},$ $\Compl{S} \cap T,$ $\Compl{S} \cap \Compl{T}$ are non-empty, otherwise $\Cut{S}$ and $\Cut{T}$ are called \emph{laminar}.
\Cref{thm:lovasztightcuts} states that all maximal families of pairwise laminar tight cuts of a matching covered graph $G$ yield the same list of bricks and braces.

Many matching theoretical concepts can be expressed in terms of matching minors.
Especially in bipartite graphs they play a huge role as the following results by Lucchesi et al.~illustrate.
	
\begin{lemma}[\cite{lucchesi2015thin}]\label{lemma:lucchesimatminor}
	Let $G$ be a bipartite matching covered graph and $\Cut{Z}$ a \nontrivial tight cut in $G$.
	Then the two $Z$-contractions of $G$ are matching minors of $G$.
\end{lemma}

\begin{corollary}[\cite{lucchesi2015thin}]\label{cor:lucchesimatminor}
	Every brace of a bipartite matching covered graph $G$ is a matching minor of $G$.
\end{corollary}

\section{Perfect matching width, tight cuts and matching minors}
\label{sec:basicproperties}

The main objective of this section is to prove \cref{thm:pmwandmatminors,thm:pmwandbricksandbraces}.
To this end we start by considering the relation between the parity of vertex sets and the decomposition tree in a perfect matching decomposition.
We then consider tight cut contractions and establish that the maximum perfect matching width among the bricks and braces of a matching covered graph yields an upper bound on the perfect matching width of the graph itself.
This establishes one of the inequalities in \cref{thm:pmwandbricksandbraces} as it proves that conformal subgraphs cannot have larger perfect matching width than the graph itself.
For the remaining part, namely the contractions, we make use of $M$-decompositions, which cause the factor 2 in \cref{thm:pmwandmatminors,thm:pmwandbricksandbraces}.

	\subsection{Cubic trees and perfect matching decompositions of even width}
	\label{sec:trees}
	
	The trees of perfect matching decompositions, as for many branch decompositions (see \cite{vatshelle2012new} for the general definition), are cubic, or at least subcubic.
	Just considering the possible structures of the trees themselves can be a useful tool when dealing with this kind of decompositions.
	We take a closer look at the cubic trees that appear as the trees of perfect matching decompositions.
	
	For a cubic tree $T$ we define the \emph{spine} of $T$ by $\Spine{T}\coloneqq T-\Leaves{T}$.
	The edges in $\Fkt{E}{T}\setminus\Fkt{E}{\Spine{T}}$ are called \emph{trivial}.
	We say that an edge $e\in\Fkt{E}{\Spine{T}}$ is \emph{even}, if the two components of $T-e$ contain an even number of leaves of $T$ each and it is \emph{odd} otherwise.
    See \cref{fig:spine_and_odd_edges} for an illustration of these definitions.

 \begin{figure}[!ht]
	\centering
	\begin{tikzpicture}
		\def\dist{0.7}
		\node[vertex] (center) at (0,0) {};
		\node[vertex] (l) at ($(center)+(-\dist,0)$) {};
		\node[vertex] (r) at ($(center)+(\dist,0)$) {};
		\node[vertex] (s) at ($(l)+(270:\dist)$) {};
		\node[vertex] (x) at ($(l)+(180:\dist)$) {};
		\node[vertex] (g) at ($(x)+(135:\dist)$) {};
		\node[vertex] (h) at ($(x)+(225:\dist)$) {};
		
		\node[vertex] (z) at ($(r)+(270:\dist)$) {};
		\node[vertex] (y) at ($(r)+(360:\dist)$) {};
		\node[vertex] (d) at ($(y)+(45:\dist)$) {};
		\node[vertex] (e) at ($(y)+(315:\dist)$) {};
		
		\node[vertex] (o) at ($(center)+(90:\dist)$) {};
		\node[vertex] (b) at ($(o)+(135:\dist)$) {};
		\node[vertex] (c) at ($(o)+(45:\dist)$) {};
		
		\draw (y) edge[markedOnly={myLightBlue}] (x);
		\draw (center) edge[edge,myOrange] (l);
		\draw (r) edge[edge,myOrange] (center);
		\draw (o) edge[marked={myLightBlue},shorten >= 2pt] (center);
		\draw (o) edge[edge] (center);
		\draw (b) edge[edge] (o);
		\draw (c) edge[edge] (o);
		\draw (z) edge[edge] (r);
		\draw (y) edge[edge] (r);
		\draw (d) edge[edge] (y);
		\draw (e) edge[edge] (y);
		\draw (g) edge[edge] (x);
		\draw (h) edge[edge] (x);
		\draw (x) edge[edge] (l);
		\draw (l) edge[edge] (s);
	\end{tikzpicture}
	\caption{An example for a cubic tree $T$ with its \protect\hyperlink{spine}{\textcolor{myLightBlue!80!black}{spine}} and its \textcolor{myOrange}{odd edges}.}
	\label{fig:spine_and_odd_edges}
\end{figure}

	Note that, if $T$ is the cubic tree underlying a perfect matching decomposition of a graph $G$, then $T$ has an even number of leaves as $G$ has an even number of vertices.
	This implies that in $T$ a \nontrivial edge $e$ is odd if and only if the two trees of $T-e$ contain an odd number of leaves of $T$ each.
	
	We make the following useful observations on cubic trees with an even number of leaves.
	
	\begin{observation}
    \label{lemma:cubictrees}
		Let $T$ be a cubic tree with $\Abs{\Leaves{T}}=\ell$ even.
		Then the following statements are true.
		\begin{enumerate}
			
			\item $\Abs{\Fkt{V}{T}}=2\ell-2$,
			
			\item $\Spine{T}$ has an even number of vertices,
			
			\item $\Spine{T}$ has an even number of vertices of degree $2$, and
			
			\item $e\in\Fkt{E}{\Spine{T}}$ is an odd edge of $T$ if and only if the two trees of $\Spine{T}-e$ contain an even number of vertices each.
			
		\end{enumerate}
	\end{observation}

	It is easy to see that $\Spine{T}$ is a subcubic tree.
	There is a close correspondence between the occurrence of odd edges in $T$ and vertices of degree $2$ in $\Spine{T}$.
	
	\begin{lemma}\label{lemma:oddspinevertices}
		Let $T$ be a cubic tree with an even number of leaves.
		\begin{enumerate}
		
			\item If $\DegG{v}{\operatorname{spine}\Brace{T}}=1$, then $v$ is not incident with an odd edge of $T$.
			
			\item If $\DegG{v}{\operatorname{spine}\Brace{T}}=2$, then $v$ is incident with exactly $1$ odd edge of $T$.
			
			\item If $\DegG{v}{\operatorname{spine}\Brace{T}}=3$, then $v$ is either incident with exactly $2$ odd edges of $T$ or with none.
		
		\end{enumerate}
	\end{lemma}
	
	\begin{proof}
	If $v$ is of degree $1$ in the spine of $T$, it is adjacent with exactly two leaves of $T$ and thus, by definition, the unique edge incident with $v$ in $\Spine{T}$ cannot be odd.

	Let $v$ be a vertex of degree $2$ in $\Spine{T}$ and $e_1,e_2$ the two edges incident with $v$ in the spine.
	In $T$ itself $v$ is incident with a third edge $e_3$ whose other endpoint is a leaf of $T$.
	Let $k_i$ be the number of leaves of $T$ contained in the component of $T-e_i$ that does not contain $v$.
	Then $\Abs{\Leaves{T}}=k_1+k_2+k_3$ and $k_3 = 1$.
	Since the total number of leaves is even and $k_3$ is odd, exactly one of $k_1$ and $k_2$ is odd as well.
	Thus, exactly one of the two edges $e_1$ and $e_2$ is an odd edge of $T$.
	
	At last we consider a degree $3$ vertex $v$ in $\Spine{T}$.
	Let $e_1,e_2,e_3$ be the three edges of the spine incident with $v$ and let $k_i$ be the number of leaves of $T$ contained in the component of $T-e_i$.
	In this case $\Abs{\Leaves{T}}=k_1+k_2+k_3$ and thus neither all three, nor just one of them can be odd.
	\end{proof}

	In particular, no vertex in $T$ can be incident with more than two odd edges and every degree $2$ vertex of $\Spine{T}$ is incident with exactly one odd edge of $T$.
	Hence, the following corollary holds.
	
	\begin{corollary}\label{cor:cubicspine}
		Let $T$ be a cubic tree with an even number of leaves.
		Then $\Spine{T}$ is cubic if and only if $T$ has no odd edges.
	\end{corollary}
	
	Moreover, the odd edges of $T$ induce a subforest of $\Spine{T}$ such that the leaves of this forest are exactly the degree $2$ vertices of $\Spine{T}$.
	Also, no vertex of $\Spine{T}$ can be incident with more than two odd edges of $T$ and thus this subforest is actually a collection of paths.

	\begin{corollary}\label{cor:oddspine}
	Let $T$ be a cubic tree with an even number of leaves and $E_O\subset\Fkt{E}{T}$ the set of odd edges of $T$.
	Then $\InducedSubgraph{T}{E_O}$ is a collection of pairwise disjoint paths.
	Moreover, the set of endpoints of these paths is exactly the set of degree $2$ vertices in $\Spine{T}$.
	\end{corollary}

	Next, we want to answer the question:
	How do odd edges interact with the perfect matching width?
	First, we investigate the influence of the existence of odd edges in the cubic tree of a perfect matching decomposition $\Brace{T,\DecompBijection{}}$ on the parity of the width of $\Brace{T,\DecompBijection{}}$.
	
	\begin{observation}
	    \label{lemma:parityporosity}
		Let $G$ be a graph with a perfect matching and $X \subseteq \Fkt{V}{G}$.
		Then $\MatPor{\Cut{X}}$ is odd if and only if $\Abs{X}$ is odd.
	\end{observation}

As an immediate consequence of \cref{lemma:parityporosity}, the tree of every perfect matching decomposition $\Brace{T,\DecompBijection{}}$ of odd width contains an odd edge.
Parity plays a huge role in the study of perfect matchings and it can be very useful to control the occurrence of odd edges in a perfect matching decomposition.

\subsection{Upper bound by bricks and braces}
\label{sec:tightCutsBricksBraces}

In order to establish a basic toolkit, in this subsection we present how the perfect matching width of $G$ relates to the perfect matching width of its matching minors.
We establish that the bricks and braces of $G$ yield an upper bound on the perfect matching width of $G$ itself.
This connection reduces the problem of finding a (close to optimal) perfect matching decomposition of $G$ to finding appropriate decompositions for its bricks and braces.

	\begin{proposition}\label{prop:upperboundbricksandbraces}
		Let $G$ be a matching covered graph.
		Then
		\begin{equation*}
			\pmw{G}\leq \max_{\substack{H~\text{brick or} \\ \text{brace of }G}} \pmw{H}.
		\end{equation*}
	\end{proposition}
    \begin{proof}
        Due to \cref{thm:lovasztightcuts} it suffices to show that for every matching covered graph $G$ and every tight cut $\Cut{Z}$ holds $\pmw{G} \leq \max\Set{\pmw{G_Z}, \pmw{G_{\Compl{Z}}}}$.
        
        To this end let $\Brace{T_i , \DecompBijection{}_i }$ be an optimal perfect matching decomposition of $G_i$, for $i \in\Set{Z,\Compl{Z}}$ and let $\ell_i$ be the leaf of $T_i$ such that $\Fkt{\DecompBijection{}_i}{\ell_i} = v_{j}$, where $j \in\Set{Z,\Compl{Z}}\setminus\Set{i}$, that is, $v_{j}$ is the contraction vertex in $G_i$.
        
        We construct a perfect matching decomposition $\Brace{T, \DecompBijection{}}$ of $G$ as follows.
        Let $e_i$ be the edge of $T_i$ that is incident to $\ell_i.$
        First we obtain $T$ from $T_Z$ and $T_{\Compl{Z}}$ by identifying the edges $e_Z$ and $e_{\Compl{Z}}$ as a new edge $e$.
        \begin{align*}
		\DecompBijection{} &{}: \Leaves{T} \to \Fkt{V}{G},\\
        \DecompBijection{t} &{}\coloneqq 
                \begin{cases}
                    \Fkt{\DecompBijection{}_{Z}}{t}, &\text{ if } t\in\Fkt{V}{T_Z}\setminus\Set{\ell_{Z}}\\
                    \Fkt{\DecompBijection{}_{\Compl{Z}}}{t}, &\text{ if } t\in\Fkt{V}{T_{\Compl{Z}}}\setminus\Set{\ell_{\Compl{Z}}}
                \end{cases}
		\end{align*}
    
    We claim that the width of $\Brace{T, \delta}$ is equal to the maximum of the widths of $\Brace{T_Z , \DecompBijection{}_Z }$ and $\Brace{T_{\Compl{Z}} , \DecompBijection{}_{\Compl{Z}} }$.
    As $\Cut{Z}$ is a tight cut, $\Cut{e}$ has matching porosity $1$.
	Now consider an edge $e' \in \Fkt{E}{T} \setminus \Set{e}$.
	Let $i \in \Set{Z,\Compl{Z}}$ such that $e' \in \Fkt{E}{T_i}$ and let $j \in\Set{Z,\Compl{Z}}\setminus\Set{i}$.
    For every perfect matching $M \in \Perf{G}$
    the $\Remainder{M}{G_i}$ is a perfect matching of $G_i$ with $\Abs{\CutG{G_i}{e'}\cap M'} = \Abs{\CutG{G}{e'}\cap M}$.
    Thus $\MatPor{\CutG{G}{e'}} = \MatPor{\CutG{G_i}{e'}}$.
    
    Therefore, $\pmw{G}$ is at most $\max\Set{\pmw{G_X},\pmw{G_{\Compl{X}}}}$
    \end{proof}

For the study of matching covered graphs of specific perfect matching width it would be helpful to have a notion of obstructions, or at least sources for lower bounds, on the width.
Before we continue towards the main result of this section concerning matching minors, we have to discuss conformal subgraphs.
These provide a lower bound on the perfect matching width of a graph and therefore are a first step in that direction.

\begin{lemma}\label{lemma:subgraphwidth}
Let $G$ be a graph with a perfect matching and $H\subseteq G$ a conformal subgraph of $G$.
Then $\pmw{H}\leq\pmw{G}$.
\end{lemma}

\begin{proof}
Let $\Brace{T,\DecompBijection{}}$ be an optimal perfect matching decomposition of $G$ and \begin{equation*}
    L_{\Compl{H}} \coloneqq \CondSet{\ell\in\Leaves{T}}{\DecompBijection{\ell}\in\Fkt{V}{G}\setminus\Fkt{V}{H}}.
\end{equation*}
Then, $T-L_{\Compl{H}}$ is a subcubic tree.
Now, remove from $T-L_{\Compl{H}}$ iteratively all vertices that became leaves and thus are not mapped to any vertex by $\DecompBijection{}$ obtaining $T''$.
Let $T'$ be the tree obtained from $T''$ by trimming.
We define $\DecompBijection{}'\colon\Leaves{T'}\to\Fkt{V}{H}$ by restricting $\DecompBijection{}$ to $\Leaves{T'}$ and claim that $\Brace{T',\DecompBijection{}'}$ is a perfect matching decomposition of $H$ of width at most $\pmw{G}$.

Now consider an edge $e\in\Fkt{E}{T'}$ and its corresponding cut $\Cut{X'}$ in $H.$
Let $M'\in\Perf{H}$ be a perfect matching in $H.$
By construction, $e\in\Fkt{E}{T}$ and thus $e$ corresponds to a cut $\Cut{X}$ in $G$ as well.
Moreover, $X'\subseteq X$ and $\Fkt{V}{H}\setminus X'\subseteq\Fkt{V}{G}\setminus X$.
Since $H$ is a conformal subgraph of $G$, there is a perfect matching $M\in\Perf{G}$ with $M'\subseteq M$ and we obtain $\Abs{\Cut{X'}\cap M'} \leq \Abs{\Cut{X}\cap M} \leq \pmw{G}$.
Hence, $\Width{T',\DecompBijection{}'}\geq\pmw{G}$.
\end{proof}

We want to reduce the problem of determining the perfect matching width of a matching covered graph to working out the width of its bricks and braces.
For now we established that the width of the bricks and braces yields an upper bound and conformal subgraphs provide a lower bound.
In order to obtain a lower bound in terms of the bricks and braces, we need to know how tight cuts interact with the perfect matching width.

\subsection{\texorpdfstring{$M$}{M}-perfect matching width and matching minors}
\label{sec:M-PMW}

For bipartite matching covered graphs Rabinovich and two of the authors \cite{hatzel2019cyclewidth} provide a qualitative bound for the perfect matching width of matching minors.
In this section we strengthen this to general graphs with perfect matchings.
We use the notion of $M$-perfect matching width ($M$-pmw), which allows us to restrict ourselves to a specific kind of perfect matching decompositions.
{\renewcommand\footnote[2]{}\mpmwandpmw*}

We start with tight cut contractions.
Given an $M$-decomposition for $G$ of width $k$, we want to construct $M$-decompositions of width at most $k$ for both tight cut contractions of a single tight cut in $G$.
Handling a single tight cut contraction suffices since the $M$-decompositions we obtain for the two tight cut contractions are $M'$-perfect matching decompositions again where $M'$ is the restriction of $M$ to the two contractions.
This allows us to apply induction and reduce the initial matching covered graph $G$ all the way down to its bricks and braces.

\paragraph{Positioning the contraction vertex} Key to obtaining an $M$-decomposition for a tight cut contraction of $G$ from an $M$-decomposition $\Brace{T,\DecompBijection{}}$ of $G$ is the decision where in the trimmed version of the decomposition tree to attach a new leaf for the contraction vertex.
If there is an edge in $T$ that separates the vast majority of the vertices of one of the tight cut shores from the vertices of the other, this decision is not too complicated to make.
But if such an edge does not exists, or in other words $\Brace{T,\DecompBijection{}}$ does not distinguish between the two shores of our tight cut, it is way harder to decide.
While \cref{prop:upperboundbricksandbraces} shows that there always exist perfect matching decompositions with edges reflecting the tight cuts, these decompositions are not necessarily optimal and at this point we are not able to provide a bound on the approximation they provide.

Our decision where to position the contraction vertex is based on some implications of \cref{lemma:parityporosity}.
If $\Cut{Z}$ is a \nontrivial tight cut of $G$, then $\Abs{Z}$ is odd and thus for all $X \subseteq \V{G}$ the cut $\Cut{X}$ of $G$ has exactly one shore that contains an odd number of vertices of $Z$.
If $\Abs{X}$ is even, this shore also contains an odd number of vertices of $\Compl{Z}$.
This observation leads us to the following lemma.

Note that any cut induced by an inner edge of an $M$-decomposition is even since both shores are $M$-conformal.

\begin{lemma}\label{lemma:placecontractionvertex}
    Let $G$ be a matching covered graph, $X$ be an even cut, and $\Cut{Z}$ a \nontrivial tight cut of $G$ as well as $v_Z$ the contraction vertex obtained by the tight cut contraction of $Z$ into the graph $G_Z$.
    
    If $\Abs{X \cap Z}$ is odd, then $\MatPor{\CutG{G_Z}{\Brace{X\setminus Z} \cup \Set{v_Z}}} \leq \MatPor{\CutG{G}{X}}$.
\end{lemma}
\begin{proof}
    Define $X_1 \coloneqq X \cap Z$ and $X_2 \coloneqq \Compl{X}\cap \Compl{Z}.$
    Note that the cut $\CutG{G_Z}{\Brace{X\setminus Z} \cup \Set{v_Z}}$ in $G_Z$ corresponds to the cut $\CutG{G}{X_2}$ in $G.$
    Now consider a perfect matching $M$ of $G.$
    As $\Abs{X_1}$ is odd, we obtain
    \begin{align*}
        &1+ \Abs{M \cap \CutG{G}{X_2}} \leq \Abs{M \cap \CutG{G}{X_1}} + \Abs{M \cap \CutG{G}{X_2}} \\
        \leq{} &\Abs{M \cap \CutG{G}{X}} + \Abs{M \cap \CutG{G}{Z}} = \Abs{M \cap \CutG{G}{X}} + 1. \qedhere
    \end{align*}
\end{proof}

If $\Brace{T,\DecompBijection{}}$ is an $M$-decomposition of $G$, then, as we have seen in the proof of \cref{lemma:placecontractionvertex}, the only cuts whose matching porosity can exceed the width of $\Brace{T,\DecompBijection{}}$ by placing the contraction vertex and ``keeping'' the rest of the decomposition as it is are those of matching porosity exactly $\Width{T,\DecompBijection{}}$.

For each of those cuts we need to indicate which of its two shores contains an odd number of vertices of a tight cut shore.
To this end, we define the following orientation of the edges of $T$.
Our definition does not require $\Brace{T,\DecompBijection{}}$ to be an $M$-decomposition, however, in case it is, we are able to make further observations.

\paragraph{Z-orientations} Let $G$ be a matching covered graph, $\Cut{Z}$ a \nontrivial tight cut of $G$ and $\Brace{T,\DecompBijection{}}$ a perfect matching decomposition of $G$.
We define the \emph{$Z$-orientation} $\ZOrientation{T}{Z}$ of $T$ as the orientation of the edges of $T$, such that for every edge $t_1t_2\in\Fkt{E}{\Spine{T}}$, $\Brace{t_1,t_2}\in\Fkt{E}{\ZOrientation{T}{Z}}$ if and only if $\Abs{\DecompBijection{T_{t_2}}\cap Z}$ is odd.
Additionally, every edge $\ell t\in\Fkt{E}{T}$, where $\ell$ is a leaf, is oriented away from $\ell$, that is $\Brace{\ell,t}\in\Fkt{E}{\ZOrientation{T}{Z}}$.
Note that the $Z$-orientation of the edge $t_1t_2$ is well defined since $\Abs{Z}$ is odd (see \cref{fig:zorientation} for an example).
If there is a vertex $t\in\Fkt{V}{\ZOrientation{T}{Z}}$ such that at least two of its incident edges are outgoing edges, we call $t$ an \emph{inconsistency}.

The idea is that $\ZOrientation{T}{Z}$ should tell us where to put the contraction vertex in order to obtain a decomposition of the tight cut contraction of $G$ obtained by contracting $Z$.
However, this only works if $\ZOrientation{T}{Z}$ has no inconsistencies.

In a directed graph a vertex with only incoming edges is called a \emph{sink}.
If a $Z$-orientation does not have any inconsistencies, there exists a unique sink vertex $s$ in $\ZOrientation{T}{Z}$.
We next prove that the $Z$-orientation of an $M$-decomposition does not contain any inconsistencies and additionally, $s$ is adjacent to a leaf $t\in\Fkt{V}{T}$ and $\DecompBijection{t}\in Z$ (see \cref{lemma:uniquesink}).

So, to obtain a perfect matching decomposition of the tight cut contraction obtained from $G$ by contracting $Z$ into a single vertex $v_Z$, we now forget all vertices of $Z$, delete the corresponding leaves from $T$ (except for $t$) and map $t$ to the contraction vertex $v_Z$.
Finally, we trim this new tree.
This not only yields a perfect matching decomposition, the width of this new decomposition is at most the width of the original graph.
If our decomposition was an $M$-decomposition in the first place, we can make even stronger observations (see \cref{fig:zorientation} for an example).

\begin{figure}[!ht]
	\begin{center}
		\input{figures/zorientation}
	\end{center}
	\caption{A graph $G$ with the \nontrivial tight cut $\Cut{Z}$, a perfect matching $M\in\Perf{G}$, and an $M$-decomposition $\Brace{T,\DecompBijection{}}$ of width four.
		The arrows in $T$ are the edges forming $\ZOrientation{T}{Z}$, note that it is free of inconsistencies and has a unique sink $s$.}
	\label{fig:zorientation}
\end{figure}

\begin{lemma}\label{lemma:uniquesink}
Let $G$ be a matching covered graph, $\Cut{Z}$ a \nontrivial tight cut in $G$, $M\in\Perf{G}$, and $\Brace{T,\DecompBijection{}}$ an $M$-decomposition of $G$.
Then, $\ZOrientation{T}{Z}$ is free of inconsistencies and has a unique sink that is adjacent to a leaf.
\end{lemma}

\begin{proof}
As $Z$ defines a \nontrivial tight cut, there is a unique edge $xy\in M$ with $x\in Z$ and $y\in\Compl{Z}$.
All other vertices of $Z$ are matched within $Z$.
In the $M$-decomposition $\Brace{T,\DecompBijection{}}$ for every $t_1t_2\in\Fkt{E}{T-\Leaves{T}}$ the unique subtree $T_i$, $i \in \Set{t_1,t_2}$, with $\Abs{\DecompBijection{T_i}\cap Z}$ being odd is exactly the one that contains $x$.
Therefore, in $\ZOrientation{T}{Z}$ every inner edge is oriented towards the subtree that contains $x$ and thus there cannot be an inconsistency as $\DecompBijection{}$ is a bijection and the tree containing $x$ is well defined for every inner edge.

Moreover, let $t\in\Fkt{V}{\ZOrientation{T}{Z}}$ such that $t$ is adjacent to two leaves $\ell_1$ and $\ell_2$ where $\DecompBijection{\ell_1}=x$.
Then, for every vertex $t'\in\Fkt{V}{\ZOrientation{T}{Z}}\setminus\Set{t,\ell_1,\ell_2}$ there is a directed path in $\ZOrientation{T}{Z}$ from $t'$ to $t$.
And by the definition of $Z$-orientations, $\Brace{\ell_1,t},\Brace{\ell_2,t}\in\Fkt{E}{\ZOrientation{T}{Z}}$ which implies that $t$ is a sink of $\ZOrientation{T}{Z}$ and no vertex apart from $t$ can be a sink.
\end{proof}

Note that, in the proof above, $\DecompBijection{\ell_2}=y$ and thus, in the decomposition for the tight cut contraction we construct from $\Brace{T,\DecompBijection{}}$, the contraction vertex and $y$ are again siblings.
So, if we start out with an $M$-decomposition of a matching covered graph $G$, then the $Z$-orientations of said decomposition behave exactly as intended.
This allows us to obtain new decompositions for tight cut contractions of at most the same width and thus yields the following result.

\begin{proposition}\label{thm:tightcutctontractionsbound}
Let $G$ be a matching covered graph, $\Cut{Z}$ a \nontrivial tight cut in $G$, $M\in\Perf{G}$, and $\Brace{T,\DecompBijection{}}$ an $M$-decomposition of $G$ of width $k$.
Moreover, let $G_Z$ be the matching covered graph obtained from $G$ by contracting $Z$ into the vertex $v_Z$.
Then, there is an \hyperlink{remainderDef}{$\Remainder{M}{G_Z}$}-perfect matching decomposition of $G_Z$ of width at most $k$.
\end{proposition}

\begin{proof}
We consider the $Z$-orientation $\ZOrientation{T}{Z}$ of $T$.
By \cref{lemma:uniquesink}, $\ZOrientation{T}{Z}$ is free of inconsistencies and has a unique sink $s$.
Moreover, as we have seen, $s$ is adjacent to two leaves $t_x$ and $t_y$ of $T$ such that $\DecompBijection{t_x}=x\in Z$, $\DecompBijection{t_y}=y\in\Compl{Z}$ and $xy\in M$ is the unique edge of $M$ in $\Cut{Z}$.

We now construct a perfect matching decomposition $\Brace{T',\DecompBijection{}'}$ for $G_Z$.
To this end, let $L_Z\coloneqq\CondSet{t\in\Leaves{T}}{\DecompBijection{t}\in Z\setminus\Set{x}}$ and $T''$ be the tree obtained from $T-L_Z$ by repeatedly removing leaves that do not lie in the domain of $\nu$.
Then, let $T'$ be the cubic tree obtained from $T''$ by trimming.
Then $\Leaves{T'}=\Leaves{T}\setminus L_Z$ and for every $t\in\Leaves{T'}$ and every inner edge $t_1t_2\in\Fkt{E}{T'}$, $t$ is a leaf of the tree $T_i'$, $i \in \Set{t_1,t_2}$, if and only if $t$ is a leaf of the subtree $T_i$.
Therefore, every bipartition of $\Leaves{T}$ induced by an inner edge of $T'$ is also induced by an edge in $T$.

To obtain $\DecompBijection{}'$ from $\DecompBijection{}$ we do not change anything for $\Compl{Z}$ and replace $x$ by $v_Z$. So for all $t\in\Leaves{T'}$ let
\begin{equation*}
	\Fkt{\DecompBijection{}'}{t}\coloneqq\TwoCases{v_Z}{\text{if}~\DecompBijection{t}=x,~\text{and}}{\DecompBijection{t}}{\text{otherwise.}}
\end{equation*}

The restriction $\Remainder{M}{G_Z}$ of $M$ to $G_Z$ contains all edges with both endpoints in $\Compl{Z}$ and additionally the edge $yv_Z$, so by construction, $\Brace{T',\DecompBijection{}'}$ it is an $\Remainder{M}{G_Z}$-perfect matching decomposition of $G_Z$.

Now, let $t_1t_2\in\Fkt{E}{T'}$ be an inner edge and $\CutG{G}{t_1t_2}$ the cut induced by $t_1t_2$ in $G$ via $\Brace{T,\DecompBijection{}}$.
Then, $\CutG{G}{t_1t_2}$ has a unique shore $X\subseteq\Fkt{V}{G}$ that contains $x$ and, as $\Brace{T,\DecompBijection{}}$ is an $M$-decomposition, $\Abs{X}$ is even.
Moreover, the cut $\CutG{G_Z}{t_1t_2}$ induced by $t_1t_2$ in $G_Z$ via $\Brace{T',\DecompBijection{}'}$ has a shore $X' \coloneqq \Brace{X\setminus{Z}}\cup\Set{v_Z}$.
As $\Abs{X\cap Z}$ is odd by choice of $x$, \cref{lemma:placecontractionvertex} gives us $\MatPor{\CutG{G_Z}{X'}}\leq\MatPor{\CutG{G}{X}}\leq k$ and thus concludes this proof.
\end{proof}

Since \cref{thm:tightcutctontractionsbound} provides an $\Remainder{M}{G_Z}$-perfect matching decomposition of the tight cut contraction $G_Z$, we can now choose a new tight cut in $G_Z$ and continue with a new iteration of the tight cut decomposition procedure.
So finally, we reach decompositions of the bricks and braces of $G$ with width still bound by the width of the original $M$-decomposition.
By then applying \cref{thm:Mpmwandpmw} we obtain the following corollary.

\begin{corollary}\label{cor:bricksandbraceslowerbound}
Let $G$ be a matching covered graph and $H$ a brick or brace of $G$.
Then $\pmw{H}\leq 2\pmw{G}$.
\end{corollary}

So by iteratively contracting tight cuts we cannot significantly increase the perfect matching width.
As bicontractions are a special case of tight cut contractions and by \cref{lemma:subgraphwidth} the width of a conformal subgraph of $G$ is bounded by the width of $G$ itself, we finally obtain our main result on matching minors of this section.

\pmwmatminor*

\Cref{cor:bricksandbraceslowerbound} and \cref{prop:upperboundbricksandbraces} now yield the wanted relation between the perfect matching width of the bricks and braces and the width of the graph itself.
\pmwbricksbraces*

Moreover, if we consider the $M$-perfect matching width of a matching covered graph $G$, we obtain an even stronger result which concludes this section.

\begin{restatable}{lemma}{mpmwmatchingminorlemma}\label{lemma:Mpmwandmatminors}
Let $G$ be a matching covered graph with a perfect matching $M$ and $H$ be a matching minor of $G$ obtained from an $M$-conformal subgraph of $G$ by a series of bicontractions, or a brick or brace of $G$.
Then $\Mpmw{\Remainder{M}{H}}{H}\leq\Mpmw{M}{G}$.
\end{restatable}

\section{Braces of Perfect Matching Width 2}\label{sec:pmw2}

The only matching covered graph of perfect matching width one is $K_2$.
Apart from this every perfect matching decomposition of a matching covered graph contains a vertex that is adjacent to two leaves (which, by definition, are mapped to two distinct vertices of $G$) and, as $G$ is matching covered, there is a perfect matching which does not match these vertices with each other.
Therefore, the cut in $G$ induced by the non-leaf edge of said vertex in the decomposition has matching porosity two.
So two is a natural lower bound on the perfect matching width of braces.
One approach to width parameters can be to investigate the structure of graphs of small width.
Since by \cref{prop:upperboundbricksandbraces} the perfect matching width of a graph is bounded from above by the width of its bricks and braces, studying the structure of braces of perfect matching width two appears to be a good starting point towards a better understanding of the parameter itself.
We present two possible characterisations of perfect matching width two braces, one in terms of edge-maximal graphs similar to the $k$-tree characterisation of treewidth $k$ graphs (see \cite{arnborg1985efficient} for an overview on this topic) and the other one in terms of elimination orderings, which again resembles similar results on treewidth.

\subsection{Generalised Tight Cuts}

We start by introducing some additional facts about braces and the edge cuts that can be found within a brace.
Tight cuts are defined as those edge cuts which contain exactly one edge of every perfect matching of our graph.
In a similar way we may define generalised versions of these cuts.

\begin{definition}[Generalised Tight Cut]
	Let $G$ be a matching covered graph, $k\in\N$ a positive integer, and $X\subseteq\V{G}$.
	The edge cut $\CutG{G}{X}$ is \emph{$k$-tight} if $\Abs{\CutG{G}{X}\cap M}=k$ for all $M\in\Perf{G}$.
	If $\CutG{G}{X}$ is a $k$-tight cut, we say that $X$ \emph{induces} a $k$-tight cut.
	A $k$-tight cut is \emph{trivial} if $\Abs{X}= k$ or $\Abs{\Compl{X}}= k$.
\end{definition}

In the following we investigate the properties of $k$-tight cuts in bipartite graphs, and braces in particular, a bit further.

\begin{definition}[Minority and Majority]
	Let $B$ be a bipartite graph, and $X\subseteq\V{G}$.
	If $\Abs{X\cap V_1}=\Abs{X\cap V_2}$ we say that $X$ is \emph{balanced}, otherwise it is \emph{unbalanced}.
	Suppose $X$ is unbalanced, then there are $i,j\in\Set{1,2}$, and $k\in\N$ such that $\Abs{X\cap V_i}=\Abs{X\cap V_j}+k$.
	In this case we call $X\cap V_i$ the \emph{majority} of $X$, denoted by $\Majority{X}$, and $X\cap V_j$ is the \emph{minority}, denoted by $\Minority{X}$.
	We say that $k$ is the \emph{imbalance} of $X$, and in general we set
	\begin{align*}
		\Imbalance{X}\coloneqq\TwoCases{0}{\text{if $X$ is balanced, or}}{k}{\text{if the imbalance of $X$ is $k$.}}
	\end{align*}
\end{definition}

\begin{lemma}\label{lemma:ktightcolours}
	Let $B$ be a bipartite matching covered graph, $k\in\N$ a positive integer, and $X\subseteq\V{B}$ a set of vertices that induces a $k$-tight cut.
	Then there exist $k_1,k_2\in\N$ such that for every perfect matching $M\in\Perf{B}$ there are exactly $k_i$ vertices of $X\cap V_i$ which are matched by edges of $\CutG{B}{X}\cap M$ for both $i\in\Set{1,2}$.
\end{lemma}

\begin{proof}
	Let $M$ be some perfect matching of $B$ and for both $i\in\Set{1,2}$, let $k_i$ be the number of vertices in $X\cap V_i$ which are matched by edges of $\CutG{B}{X}\cap V_i$.
	Then every other edge of $M$ either has both or no endpoint in $X$.
	Hence there is a number $n\in\N$ of edges of $M$ with both endpoints in $X$ such that $\Abs{X}=k_1+k_2+2n$.
	Now suppose, towards a contradiction, there exist $k_1',k_2'\in\N$ together with a perfect matching $M'\in\Perf{B}$ such that for each $i\in\Set{1,2}$, $k_i'$ is the number of vertices of $X\cap V_i$ that are matched by edges of $\CutG{B}{X}$, and $k_1'\neq k_1$, which also implies $k_2'\neq k_2$.
	By the same arguments as before, there exists a number $n'$ such that $\Abs{X}=k_1'+k_2'+2n'$.
	Indeed, we have $\Abs{X\cap V_i}=k_i+n=k_i'+n'$ for both $i\in\Set{1,2}$.
	Without loss of generality, let us assume $k_1'>k_1$.
	Then $n'<n$ since $k_1+n=k_1'+n'$.
	But since $\CutG{B}{X}$ is $k$-tight, we have $k_1+k_2=k=k_1'+k_2'$.
	Hence
	\begin{align*}
		\Abs{X}=k+2n>k+2n'=\Abs{X},
	\end{align*}
	which is impossible and thus our claim holds.
\end{proof}

The following can be seen as a generalisation of an observation first made by Lov\'asz (see the proof of Lemma 1.4 in \cite{lovasz1987matching}).

\begin{lemma}\label{lemma:ktightmajorityminority}
	Let $B$ be a bipartite matching covered graph, $k\in\N$ a positive integer, and $X\subseteq\V{G}$ a set of imbalance $k$.
	Then $\CutG{B}{X}$ is $k$-tight if and only if $\NeighboursG{B}{\Minority{X}}\subseteq\Majority{X}$.
\end{lemma}

\begin{proof}
	Let us first assume $X$ induces a $k$-tight cut and suppose there is some edge $e\in\CutG{B}{X}$ such that $e$ has an endpoint in $\Minority{X}$.
	As $G$ is matching covered there exists $M_e\in\Perf{B}$ such that $e\in M_e$.
	Now there are $\Abs{\Minority{X}}-1$ many vertices of the minority left which can be matched by $M_e$ to vertices of the majority of $X$.
	Hence at least $k+1$ vertices of $\Majority{X}$ cannot be matched by $M_e$ with vertices inside $X$.
	This however means that $\Abs{\CutG{B}{X}\cap M_e}\geq k+2$, contradicting the assumption that $X$ induces a $k$-tight cut.
	
	For the reverse direction, let us assume $\NeighboursG{B}{\Minority{X}}\subseteq\Majority{X}$.
	Then for every $M\in\Perf{B}$, every vertex of $\Minority{X}$ must be matched with a  vertex of $\Majority{X}$, therefore leaving exactly $\Imbalance{X}=k$ vertices of $\Majority{X}$ which must be matched via edges of $\CutG{B}{X}$.
	Therefore $\Abs{\CutG{B}{X}\cap M}=k$ for all $M\in\Perf{B}$.
\end{proof}

\subsection{Spines and Cuts in \texorpdfstring{$k$}{k}-Extendable Bipartite Graphs}
\label{sec:general_k}

In this subsection we present a few results on general $k$, which we then make use of for $k=2$ over the next subsections.
We start by introducing the notion of $k$-extendability, which, in the case $k=2$ is a property equivalent to the absence of non-trivial tight cuts in bipartite matching covered graphs.

\begin{definition}
	Let $G$ be a graph with a perfect matching and $F\subseteq\E{G}$ a matching.
	We say that $F$ is \emph{extendable} if there exists $M\in\Perf{G}$ such that $F\subseteq M$.
	
	For any positive integer $k\in\N$, $G$ is said to be \emph{$k$-extendable} if it is connected, has at least $2k+2$ vertices, and every matching of size $k$ in $G$ is extendable.
\end{definition}

This notion provides a characterisation of braces.

\begin{theorem}[\cite{lovasz2009matching}]\label{thm:braces}
	A bipartite graph $B$ is a brace if and only if it is either isomorphic to $C_4$, or it is $2$-extendable.
\end{theorem}

In some sense \cref{thm:braces} generalises to higher values of extendability.

\begin{theorem}[\cite{plummer1986matching}]\label{thm:bipartiteextendability}
	Let $B$ be a bipartite graph and $k\in\N$ a positive integer.
	The following statements are equivalent.
	\begin{enumerate}
		\item $B$ is $k$-extendable.
		\item $\Abs{V_1}=\Abs{V_2}$, and for all non-empty $S\subseteq V_1$ with $\Abs{S}\leq\Abs{V_1}-k$, $\Abs{\NeighboursG{B}{S}}\geq \Abs{S}+k$.
		\item For all sets $S_1\subseteq V_1$ and $S_2\subseteq V_2$ with $\Abs{S_1}=\Abs{S_2}\leq k$ the graph $B-S_1-S_2$ has a perfect matching.
	\end{enumerate}
\end{theorem}

Additionally, we need the following properties of $k$-extendable graphs and the well-known K\H{o}nig's Theorem.

\begin{theorem}[\cite{plummer1980n}]\label{thm:smallerextendabilites}
	Let $k\in\N$ be a positive integer.
	Then every $k$-extendable graph is also $\Brace{k-1}$-extendable.
\end{theorem}

\begin{theorem}[\cite{plummer1980n}]\label{thm:extendabilitytoconnectivity}
	Let $k\in\N$ be a positive integer.
	Then every $k$-extendable graph is $\Brace{k+1}$-connected.
\end{theorem}

\begin{theorem}[{K\H{o}nig's Theorem (see for example \cite{lovasz2009matching})}]
	\label{thm:konig}
	If $B$ is a bipartite graph, then $\VCNum{B} = \MatNum{B}.$
\end{theorem}

Using these we can establish that in bipartite $k$-extendable graphs no cut of small porosity can have two large shores.
To this end we define for every graph $G$ and cut $\CutG{G}{X}$ in $G$ the graph $\InducedSubgraph{G}{\CutG{G}{X}}$ as the subgraph of $G$ induced by all the edges in $\CutG{G}{X}.$

\begin{lemma}
	\label{lem:small_porosity_implies_separator}
	Let $k \in \N$ and $B$ be a bipartite $k$-extendable graph and $X \subseteq \V{B},$ then one of the following holds for every $k' \leq k$:
	\begin{enumerate}
		\item $\MatNum{\InducedSubgraph{B}{\CutG{B}{X}}} > k',$
		\item $\Abs{X} \leq k',$ or
		\item $\Abs{\Compl{X}} \leq k'.$
	\end{enumerate}
\end{lemma}
\begin{proof}
	We assume $\MatNum{\InducedSubgraph{B}{\CutG{B}{X}}} \leq k'.$
	As $B$ is bipartite, the graph $\InducedSubgraph{B}{\CutG{B}{X}}$ is as well.
	By \cref{thm:konig} (K\H{o}nig's Theorem), we thus obtain $\VCNum{\InducedSubgraph{B}{\CutG{B}{X}}} = \MatNum{\InducedSubgraph{B}{\CutG{B}{X}}} \leq k'.$
	So there is a set of vertices $S$ of size at most $k'$ hitting all edges crossing $\CutG{B}{X}.$
	This means $S$ is a separator of size at most $k'$ in $B$ separating $X \setminus S$ from $\Compl{X} \setminus S.$
	By \cref{thm:extendabilitytoconnectivity}, $B$ is $k+1$-connected, therefore $X \subseteq S$ and thus $\Abs{X} \leq k',$ or $\Compl{X} \subseteq S$ and thus $\Abs{\Compl{X}} \leq k'.$
\end{proof}

We now have all necessary tools available to investigate the structure of optimal perfect matching decompositions of $k$-extendable bipartite graphs whose perfect matching width is close to $k$.

Since we are interested in braces, it seems natural to ask for the structure of non-trivial cuts in highly extendable bipartite graphs as one might find some connection to proper $k$-tight cuts and the relation between minority and majority as seen in \cref{lemma:ktightmajorityminority}.

In the following we denote the size of a maximum matching in a graph $G$ by $\MatNum{G}$ and the size of a minimum \emph{vertex cover} by $\VCNum{G}$, where a vertex cover of $G$ is a set of vertices $S\subseteq\V{G}$ such that every edge of $G$ has at least one endpoint in $S$.

We establish a connection between the matching porosity of a cut and its imbalance in $k$-extendable bipartite graphs.

\begin{lemma}\label{lemma:imbalance}
	Let $k\in\N$ be a positive integer, $B$ be a $k$-extendable and bipartite graph, and $X\subseteq\Fkt{V}{G}$ such that $\MatPor{\CutG{B}{X}}=k$ and $k+2\leq\Abs{X}\leq\Abs{\V{B}}-\Brace{k+2}$.
	Then $\Imbalance{X}=k$.
\end{lemma}

\begin{proof}
 We start by observing that every perfect matching of $B$ has at most $\Minority{X}$ many edges matching two vertices of $X.$
	Thus, the matching porosity of $\CutG{B}{X}$ yields an upper bound on the imbalance of $X,$ that is, $k = \MatPor{\CutG{B}{X}} \geq \Abs{X} - 2\Abs{\Minority{X}} = \Imbalance{X}.$
	By \cref{lemma:parityporosity} and as the imbalance and the size of a vertex set have the same parity, we know that $k \equiv \Abs{X} \equiv \Imbalance{X} \pmod 2.$
	
	Suppose towards a contradiction that $k' \coloneqq \Imbalance{X} \leq k-2.$
	Additionally, we may assume without loss of generality that $\Majority{X} \subseteq V_2.$
	We split $\InducedSubgraph{B}{\CutG{B}{X}}$ into the following two subgraphs.
	\begin{align*}
		B_1 &\coloneqq \InducedSubgraph{B}{ \Minority{X} \cup \Brace{V_2\setminus\Majority{X}} }, \text{~and}\\
		B_2 &\coloneqq \InducedSubgraph{B}{ \Brace{V_1\setminus\Minority{X}} \cup \Majority{X} }.
	\end{align*}
	We have $B_1 \cup B_2 = \InducedSubgraph{B}{\CutG{B}{X}}.$

	Suppose $\MatNum{B_1} \geq \frac{k-k'}{2}+1,$ then there is a matching $F$ of size $\frac{k-k'}{2}+1$ in $B_1.$
	As $\Abs{F} = \frac{k-k'}{2}+1 \leq k$ and $B$ is $k$-extendable, there is a perfect matching $M_F$ of $B$ with $F\subseteq M_F.$
	Due to $\MatPor{\CutG{B}{X}}=k,$ at most $k- \Brace{\frac{k-k'}{2}+1} = \frac{k+k'}{2}-1$ edges of $M_F \cap \CutG{B}{X}$ have an endpoint in $\Majority{X}.$
	Thus we obtain
	\begin{align*}
		\Abs{\Majority{X}\setminus\V{M_F \cap \CutG{B}{X}}} & \geq \Abs{\Majority{X}}-\Brace{\frac{k+k'}{2}-1}\\
		& = \Abs{\Majority{X}}-\frac{k}{2}-\frac{k'}{2}+1\\
		& = \Abs{\Minority{X}}+k'-\frac{k}{2}-\frac{k'}{2}+1\\
		& > \Abs{\Minority{X}} - \frac{k-k'}{2} -1\\
		&  \geq \Abs{\Minority{X} \setminus \V{M_F \cap \CutG{B}{X}}}.
	\end{align*}
	Therefore, $\Imbalance{X\setminus\V{M_F \cap \CutG{B}{X}}}\geq 1,$ contradicting $M_F$ being a perfect matching.
	Thus, $\MatNum{B_1}\leq\frac{k-k'}{2}.$
	With similar arguments $\MatNum{B_2}\geq\frac{k+k'}{2}+1$ yields a contradiction.
	Thus, $\MatNum{B_2}\leq\frac{k+k'}{2}.$
	It follows that 
	\begin{equation*}
		\MatNum{\InducedSubgraph{B}{\CutG{B}{X}}} = \MatNum{B_1} + \MatNum{B_2} \leq \frac{k-k'}{2}+\frac{k+k'}{2}=k.
	\end{equation*}
	Together with $k+2\leq\Abs{X}\leq\Abs{\V{B}}-\Brace{k+2},$ this contradicts \cref{lem:small_porosity_implies_separator}.
	Thus, we obtain that $\Imbalance{X}=k.$
\end{proof}

\Cref{lemma:imbalance} establishes the distribution of the two colours $V_1$ and $V_2$ in any set of matching porosity $k$ of sufficient size in a $k$-extendable brace.
Note that as $k$ is the matching porosity it always is even.

Next, we show that decompositions with the special structure of the spine of the spine of the decomposition tree being a path have the following property.
The edges of this path having matching porosity exactly $k$ induce shores that have imbalance $k$ and the neighbourhood of the minority of these shores is completely contained in the shore.
So the shores have a kind of closure property when it comes to the minority: no vertex of the minority has neighbours outside the shore.

\begin{theorem}
	\label{prop:minority_closed_shores}
	Let $k\geq 2$ be an integer and $B$ be a $k$-extendable bipartite graph with a perfect matching decomposition $\Brace{T,\DecompBijection{}}$ of width $k$ such that $\Spine{\Spine{T}}$ is a path.
	Then, for all $e \in \Spine{\Spine{T}}$ with $\MatPor{\CutG{B}{e}} = k,$ every shore $X$ of $\CutG{B}{e}$ satisfies
	\begin{enumerate}
		\item $\Imbalance{X}=k,$ and
		\item $\NeighboursG{B}{\Minority{X}}\subseteq X.$
	\end{enumerate}
\end{theorem}

\begin{proof}
	We first consider how the colour classes can be distributed in the shores of two cuts $X$ and $Y$ corresponding to two adjacent edges in the path $\Spine{\Spine{T}}$ such that $X \subseteq Y.$
	We claim that $X$ and $Y$ differ by exactly one vertex from each colour class.
	
	\begin{claim}
		\label{lem:preserve_majority}
		If $\Abs{\V{B}}\geq 2k+4$ and $e_1,$ $e_2$ two adjacent edges of $\Spine{\Spine{T}}$ such that $\CutG{B}{e_1}$ has a shore $X_1$ and $\CutG{B}{e_2}$ has a shore $X_2$ with $X_1 \subseteq X_2$ and $\MatPor{\CutG{B}{X_1}} = \MatPor{\CutG{B}{X_2}} = k.$
		Then, we have $\Abs{X_1\cap V_i}+1 = \Abs{X_2\cap V_i}$ for both $i\in\Set{1,2}.$
		See \cref{fig:X_1-and-X_2} for an illustration.
	\end{claim}
 
\begin{figure}[!ht]
	\centering
	\resizebox{5cm}{!}{
	\begin{tikzpicture}
		\def\dist{1}
		\node[vertex] (center) at (0,0) {};
		\node[vertex] (e1) at ($(center)+(-1.7*\dist,0)$) {};
		\node (x-l) at ($(e1)+(-\dist,0)$) {$\dots$};
		\node[vertex] (e2) at ($(center)+(1.7*\dist,0)$) {};
		\node (x-r) at ($(e2)+(\dist,0)$) {$\dots$};
		\node[vertex] (top) at ($(center)+(0,\dist)$) {};
		\node[vertexW] (tr) at ($(top)+(90-40:\dist)$) {};
		\node[vertexB] (tl) at ($(top)+(90+40:\dist)$) {};
		
		\draw[edge] (e1) -- (center) node [midway, below] {$e_1$};
		\draw[edge] (e2) -- (center) node [midway, below] {$e_2$};
		\draw[edge] (top) to (center);
		\draw[edge] (top) to (tr);
		\draw[edge] (top) to (tl);
		
		\draw[edge,myRed,bend left] ($(e1)+(90:2*\dist)$) to ($(e1)+(270:2*\dist)$);
		\draw[edge,myGreen,bend left,looseness=0.8] ($(center)+(70:2.5*\dist)$) to ($(center)+(290:2.5*\dist)$);
		
		\node at ($(e1)+(100:2*\dist)$) {\textcolor{myRed}{$X_1$}};
		\node at ($(center)+(80:2.5*\dist)$) {\textcolor{myGreen}{$X_2$}};
	\end{tikzpicture}}
	\caption{There are exactly two vertices in $X_2\setminus X_1$ and they come from different colour classes of $B.$}
	\label{fig:X_1-and-X_2}
\end{figure}
	\begin{claimproof}
		\Cref{lemma:parityporosity} and $\Spine{\Spine{T}}$ being a path together imply $\Abs{X_2}-\Abs{X_1} = 2.$
		Suppose towards a contradiction that both vertices in $X_2 \setminus X_1$ are from the same colour class, without loss of generality say $V_1,$ that is, $\Abs{X_1\cap V_1}+2=\Abs{X_2\cap V_1}.$
		Due to $\MatPor{\CutG{B}{X_2}} = k$ we know $\Abs{\Compl{X_1}} \geq k+2$ and $\Abs{X_2}\geq k + 2.$
		So, by \cref{lemma:imbalance}, we have $\Imbalance{X_1}=\MatPor{\CutG{B}{X_1}} = k,$ or $\Abs{X_1} = k$ for $X_1$ and we have $\Imbalance{X_2}=\MatPor{\CutG{B}{X_2}} = k,$ or $\Abs{\Compl{X_2}} = k.$
		
		If $\Abs{X_1} = k$ and $\Abs{\Compl{X_2}} = k,$ then $\Abs{\V{B}} = \Abs{X_1}+2+\Abs{\Compl{X_2}} = 2k+2,$ which contradicts $\Abs{\V{B}}\geq 2k+4.$
		So we have $\Imbalance{X_1}=k$ or $\Imbalance{X_2}=k.$
		Suppose only one of them holds, so without loss of generality consider the case that $\Imbalance{X_1}\leq k-2$ and $\Imbalance{X_2}=k.$
		This implies that $\Majority{X_1}\subseteq V_1$ and $\Majority{X_2}\subseteq V_1.$
		We split $\InducedSubgraph{B}{\CutG{B}{X_2}}$ into the two subgraphs
		\begin{align*}
			B_1 &\coloneqq \InducedSubgraph{B}{\Majority{X_2}\cup\Majority{\Compl{X_2}}} \text{~and}\\
			B_2 &\coloneqq \InducedSubgraph{B}{\Minority{X_2}\cup\Minority{\Compl{X_2}}}.
		\end{align*}
		We have $B_1 \cup B_2 = \InducedSubgraph{B}{\CutG{B}{X_2}},$ and $B_1$ and $B_2$ are disjoint.
		Due to $\Abs{X_1} = k$ and $\Imbalance{X_1}\leq k-2,$ we know
		$\Abs{\Minority{X_2}}=\Abs{\Minority{X_1}}=2$ and thus, $\MatNum{B_2} \leq 2.$
		
		Suppose $\MatNum{B_1} \geq k,$ that is, there is a matching $F$ of size $k$ in $B_1.$
		By \cref{thm:smallerextendabilites}, $B$ has a perfect matching $M_F$ containing $F.$
		The set $X_1 \setminus \V{F}$ contains two vertices of $V_2$ and no vertex of $V_1.$
		So $M_F$ maps both these vertices to vertices in $\Compl{X_2}.$
		Thus, $\Abs{M_F \cap \CutG{B}{X_2}} = \Abs{X_2} = k+2,$ which contradicts $\MatPor{\CutG{B}{X_2}} = k.$
		Therefore, we have $\MatNum{B_1} \leq k-1.$
		
		Similarly, suppose there is a matching $F$ of size two in $B_2.$
		Then, by \cref{thm:smallerextendabilites}, $B$ has a perfect matching $M_F$ containing $F.$
		The set $X_2 \setminus \V{F}$ contains $k$ vertices of $V_1$ and no vertex of $V_2.$
		So $M_F$ maps all these $k$ vertices to vertices in $\Compl{X_2}.$
		Thus, $\Abs{M_F \cap \CutG{B}{X_2}} = \Abs{X_2} = k+2,$ which contradicts $\MatPor{\CutG{B}{X_2}} = k.$
		Therefore, we have $\MatNum{B_2}\leq 1.$
		
		So we obtain $\MatNum{\InducedSubgraph{B}{\CutG{B}{X_2}}} = \MatNum{B_1} + \MatNum{B_2} \leq k,$ which together with $\Abs{X_2}\geq k + 2$ and $\Abs{\Compl{X_2}}\geq k + 2$ contradicts \cref{lem:small_porosity_implies_separator}.
		
		Thus, we know that $\Imbalance{X_1} = \Imbalance{X_2} = k.$
		But this contradicts that the two vertices in $X_2 \setminus X_1$ come from the same colour class.
		So, we obtain $\Abs{X_1\cap V_i}+1=\Abs{X_2\cap V_i}$ for both $i\in\Set{1,2}.$
	\end{claimproof}

	We define $P \coloneqq \Spine{\Spine{T}}$ with the two endpoints $p_{\triangleleft}$ and $p_{\triangleright}.$
	We order the edges $\Brace{e_1,\dots,e_{\ell}}$ of $P$ by occurrence along $P$ when traversing it from $p_{\triangleleft}$ to $p_{\triangleright}.$
	Next, we show that $P$ contains two edges that induce a cut with one shore building a star with $k+1$ leaves.

	\begin{claim}
		\label{lem:end_claws}
		The tree $T$ contains edges $e_{\triangleleft}$ and $e_{\triangleright}$ such that for $\diamond \in \Set{\triangleleft,\triangleright}$ the cut $\CutG{B}{e_{\diamond}}$ has a shore $X_{\diamond}$ of size $k+2$ satisfying the following conditions.
		\renewcommand{\labelenumi}{\textbf{\theenumi}}
		\renewcommand{\theenumi}{(\roman{enumi})}
		\begin{enumerate}[labelindent=0pt,labelwidth=\widthof{\ref{last-item-k-ext-claim-2}},itemindent=1em]
			\item \label{k-ext-claim-2:s=t} $e_{\triangleleft} = e_{\triangleright}$ if and only if $\Abs{\V{B}} = 2k+4,$
			\item \label{k-ext-claim-2:disjoint} $X_{\triangleleft} \cap X_{\triangleright}=\emptyset,$
			\item \label{k-ext-claim-2:different_colours} if $\Majority{X_{\triangleleft}}\subseteq V_{i},$ then $\Minority{X_{\triangleright}}\subseteq V_{\Abs{3-i}}$ for $i \in \Set{1,2},$ and
			\item \label{k-ext-claim-2:stars} $\InducedSubgraph{B}{X_{\diamond}}$ is a star such that its central vertex has no neighbour in $\Compl{X_{\diamond}}$ for both $\diamond \in \Set{\triangleleft,\triangleright}.$ \label{last-item-k-ext-claim-2}
		\end{enumerate}
	\end{claim}
	\begin{claimproof}
		Choose $j$ minimal with $\MatPor{\CutG{B}{e_j}} = k.$
		Let $T_{\triangleleft}$ be the component in $T - e_j$ that contains $p_{\triangleleft}$ and $X'_{\triangleleft} \coloneqq \DecompBijection{T_{\triangleleft}}.$
		
		Consider the size of $X'_{\triangleleft}.$
		Due to \cref{lemma:parityporosity}, we can defer $\Abs{X_{\triangleleft}'} \neq k+1.$
		Suppose towards a contradiction that $\Abs{X'_{\triangleleft}} \geq k+2.$
		By \cref{lemma:imbalance}, this implies $\Imbalance{X'_{\triangleleft}} = k.$
		Consider the shore $X''_{\triangleleft}$ of $\CutG{B}{e_{j-1}}$ with $X''_{\triangleleft} \subseteq X'_{\triangleleft}.$
		By minimality of $j,$ we know that $\MatPor{\CutG{B}{e_{j-1}}} \leq k-1,$ and thus, $\Imbalance{X''_{\triangleleft}} \leq k-1.$
		By $\Abs{X'_{\triangleleft}} - \Abs{X''_{\triangleleft}} \leq 2,$ we know that $\Abs{X''_{\triangleleft}} \geq k,$ and by $\MatPor{\CutG{B}{X'_{\triangleleft}}} = k,$ we obtain $\Abs{\Compl{X''_{\triangleleft}}} \geq k$ as well.
		By \cref{lem:small_porosity_implies_separator}, this implies that $\CutG{B}{X''_{\triangleleft}}$ contains a matching $F$ of size $k.$
		As $B$ is $k$-extendable there is a perfect matching $M_F$ of $B$ containing $F.$
		This yields a contradiction to $\MatPor{\CutG{B}{X''_{\triangleleft}}} \leq k-1.$
		Thus, $\Abs{X'_{\triangleleft}} = k.$
		
		We define $e_{\triangleleft} \coloneqq e_{j+1}$ and additionally $X_{\triangleleft}$ to be the shore of $\CutG{B}{e_{\triangleleft}}$ containing $X'_{\triangleleft}.$
		By \cref{lemma:parityporosity}, we obtain $\Abs{X_{\triangleleft}} = k+2$ as desired.
		Moreover, $\MatPor{\CutG{B}{X_{\triangleleft}}} = k = \Imbalance{X'_{\triangleleft}} = \Imbalance{X_{\triangleleft}}.$
		Thus, by \cref{lem:preserve_majority}, $\Abs{\Brace{X_{\triangleleft} \setminus X'_{\triangleleft}} \cap V_1} = \Abs{\Brace{X_{\triangleleft} \setminus X'_{\triangleleft}}\cap V_2} = 1.$
		We know that $X_{\triangleleft} \subseteq V_i$ for some $i\in \Set{1,2},$ because $\Imbalance{X_{\triangleleft}} = \Abs{X_{\triangleleft}},$ so let us assume that $i=1$ without loss of generality.
		Together with $\MatPor{\CutG{B}{X_{\triangleleft}}} = k,$ this implies that the only vertex $w$ of $V_2$ in $X_{\triangleleft}$ has no neighbours in $\Compl{X_{\triangleleft}}.$
		As $B$ is $\Brace{k+1}$-connected, by \cref{thm:extendabilitytoconnectivity}, $\NeighboursG{B}{w}= X_{\triangleleft}\setminus\Set{w}.$
		Hence, $\InducedSubgraph{B}{X_{\triangleleft}}$ forms the desired star.
		
		Now, choose $j'$ maximal with $\MatPor{\CutG{B}{e_{j'}}} = k.$
		Let $T_{\triangleright}$ be the component in $T - e_{j'}$ that contains $p_{\triangleright}$ and $X'_{\triangleright} \coloneqq \DecompBijection{T_{\triangleright}}.$
		Also, we define $e_{\triangleright} \coloneqq e_{j'-1}$ and $X_{\triangleright}$ to be the shore of $\CutG{B}{e_{\triangleright}}$ disjoint from $X_{\triangleleft}.$
		This ensures that $X_{\triangleleft} \cap X_{\triangleright} = \emptyset.$
		By symmetric arguments to the above we obtain $\Abs{X'_{\triangleright}} = k$ and $\Abs{X_{\triangleright}} = k+2.$
		Additionally, $\MatPor{\CutG{B}{X_{\triangleright}}} = k = \Imbalance{X'_{\triangleright}} = \Imbalance{X_{\triangleright}}$ and $X_{\triangleleft} \subseteq V_i$ for some $i\in \Set{1,2}.$
		Due to \cref{lem:preserve_majority}, the decomposition gains a vertex from each colour class with every edge of matching porosity $k.$
		By \cref{lem:small_porosity_implies_separator}, the edges along $P$ lying between $e_{\triangleleft}$ and $e_{\triangleright}$ all have matching porosity $k.$
		Thus, $X_{\triangleright} \subseteq V_2,$ as $\Abs{V_2} = \Abs{V_1}.$
		So, $\InducedSubgraph{B}{X_{\triangleright}}$ forms the desired star with the centre vertex $b$ being from $V_1.$
		
		Finally, $X_{\triangleleft} \cap X_{\triangleright} = \emptyset$ implies that $e_{\triangleleft} = e_{\triangleright}$ if and only if $X_{\triangleleft} \cup X_{\triangleright} = \V{B},$ that is, $\Abs{\V{B}} = 2k+4,$ because $\Abs{X_{\triangleleft}}=\Abs{X_{\triangleright}}=k+2.$
	\end{claimproof}

	Using the edges found in \cref{lem:end_claws} and the subpath $P'$ of $P$ starting with $e_{\triangleleft}$ and ending with $e_{\triangleright}$ we now conclude the proof.
	
	If $\Abs{\V{B}} = 2k+4,$ then, by \cref{lem:end_claws}, $e_{\triangleleft} = e_{\triangleright}$ and these are the only edges inducing cuts of matching porosity $k,$ thus the statement holds.
	
	So, assume that $\Abs{\V{B}} \geq 2k+6$ and $e_{\triangleleft} \neq e_{\triangleright}.$
	We prove the statement for the shores of any edge $e_i$ lying between $e_{\triangleleft}$ and $e_{\triangleright}$ assuming that the shores of the adjacent edge $e_{i-1}$ fulfil the statement, that is, assuming $\Imbalance{X_{e_{i-1}}} = k$ and $\NeighboursG{B}{\Minority{X_{e_{i-1}}}} \subseteq X_{e_{i-1}}.$
	Let us assume without loss of generality that $\Minority{X_{e_{i}}} \subseteq V_1.$
	By \cref{lem:preserve_majority}, we know that $\Imbalance{X_{e_{i}}} = \Imbalance{X_{e_{i-1}}} = k$ and there is a unique vertex $a$ in $\Brace{X_{e_{i}}\setminus X_{e_{i-1}}}\cap V_1.$
	Suppose $a$ has a neighbour $b$ in $\Compl{X_{e_{i}}}.$
	Then, there is a perfect matching $M$ of $B$ containing $ab$ and $M\cap\CutG{B}{X_{e_{i}}}$ contains at least $k+2$ edges, because \cref{lem:preserve_majority} implies $\Minority{X_{e_{i-1}}} \subseteq V_1,$ a contradiction.
	Thus, $\NeighboursG{B}{\Minority{X_{e_{i}}}}\subseteq X_{e_{i}}$ and we are done.
\end{proof}

\subsection{Perfect matching decompositions of width \texorpdfstring{$2$}{2}}
\label{sec:k=2}

We start by establishing the basic structure the spine and the spine of the spine in the perfect matching decompositions of bipartite graphs with perfect matching width have.
The first statement, proofing that the spine of these decompositions is cubic, even holds for bricks as well.

\begin{lemma}
    \label{lemma:cubicwidth2trees}
	Let $G$ be a brick or brace of perfect matching width two and $\Brace{T,\DecompBijection{}}$ be an optimal perfect matching decomposition.
	Then, $\Spine{T}$ is cubic.
\end{lemma}
\begin{proof}
	By \cref{cor:cubicspine}, it suffices to show that $T$ is free of odd edges.
	Suppose $T$ has an odd edge $t_1t_2$, then $X_i\coloneqq\DecompBijection{T_{t_i}}$ contains an odd number of vertices for $i\in\Set{1,2}$.
	Then \cref{lemma:parityporosity} implies that $\MatPor{\CutG{G}{X_1}}$ is odd.
	As the width of $\Brace{T,\DecompBijection{}}$ is $2$ and $t_1t_2$ is an inner edge of $T$, $\Abs{X_1}\geq 3$, $\Abs{X_2} \geq 3$ and $\MatPor{\CutG{G}{X_1}}=1$.
	Thus $\CutG{G}{X_1}$ must be a non-trivial tight cut of $G$ contradicting $G$ being a brick or a brace.
\end{proof}

Using our insights on imbalance \cref{lemma:imbalance}, we can now prove that there are no degree-$3$-vertices in the spine of the spine of a width-$2$-decomposition of a brace.
This means that any optimal perfect matching decomposition of a brace $B$ with $\pmw{B}=2$ has a linear structure.

\begin{proposition}
    \label{prop:linearwidth2trees}
	Let $B$ be a brace of perfect matching width two and $\Brace{T,\DecompBijection{}}$ a perfect matching decomposition of minimum width for $B$.
	Then, $\Spine{\Spine{T}}$ is a path.
\end{proposition}
\begin{proof}
	Suppose there is a vertex $t\in\V{\Spine{\Spine{T}}}$ with three neighbours $t_1$, $t_2$ and $t_3$.
	By \cref{lemma:cubicwidth2trees}, $\Spine{T}$ is cubic and so every $t_i$ is adjacent to exactly two vertices of the spine of $T$ apart from $t$.
	Moreover, each of these neighbours again has exactly two neighbours distinct from $t_i$ in $T$.
	Let $T_i$ be the component of $T-tt_i$ for $i\in\Set{1,2,3}$ that does not contain $t$ and let $X_i\coloneqq\DecompBijection{T_i}$.
	The above observations imply $\Abs{X_i}\geq 4$ for all $i\in \Set{1,2,3}$.
	As $T$ is free of odd edges by \cref{cor:cubicspine,lemma:cubicwidth2trees}, $\MatPor{X_i}=2$ and so \cref{lemma:imbalance} yields $\Imbalance{X_i}=2$.
	
	Without loss of generality we can assume that two of the three sets have an excess in $V_1$ while the last one, say $X_3$, has an excess in $V_2$.
	This holds as the case where the excesses of all three sets are of the same colour implies $\Imbalance{\V{B}}=6$, a direct contradiction to the existence of a perfect matching in $B$.
	However, even under this assumption we still obtain $\Abs{V_1}=\Abs{V_2}+2$ and thus, $\V{G}$ is not balanced.
	Since $B$ has a perfect matching, this is impossible and thus, $\Spine{\Spine{T}}$ cannot have a vertex of degree three.
\end{proof}

This allows us to apply the findings from \cref{sec:general_k}, especially the insights obtained in the proof of \cref{prop:minority_closed_shores} to braces of perfect matching width two, we obtain the structure illustrated in \cref{fig:brace-width-2-decomposition}.

\begin{figure}[!ht]
	\centering
	\begin{tikzpicture}[scale=0.8]
		\def\innerLength{1.2}
		\def\pathLength{2}
		\def\leafLength{0.9}
		\def\leafAngle{46}
		\node (t1) at (0,0) {};
		\node (t2) at ($(t1)+(\pathLength,0)$) {};
		\node (t3) at ($(t2)+(\pathLength,0)$) {};
		\node (t4) at ($(t3)+(\pathLength,0)$) {};
		\node (t5) at ($(t4)+(\pathLength,0)$) {};
		
		\draw (t1.center) edge[edge] (t2.center);
		\draw (t2.center) edge[edge] (t3.center);
		\draw (t3.center) edge[dashed,shorten >= 5pt, shorten <= 5pt] (t4.center);
		\draw (t4.center) edge[edge] (t5.center);
		
		\node (la) at ($(t1)+(360/3:\innerLength)$) {};
		\node[vertexB] (la-1) at ($(la)+(360/3-\leafAngle:\leafLength)$) {};
		\node[vertexB] (la-2) at ($(la)+(360/3+\leafAngle:\leafLength)$) {};
		\node (lb) at ($(t1)+(2*360/3:\innerLength)$) {};
		\node[vertexB] (lb-w) at ($(lb)+(2*360/3-\leafAngle:\leafLength)$) {};
		\node[vertexW] (lb-b) at ($(lb)+(2*360/3+\leafAngle:\leafLength)$) {};
		
		\draw (t1.center) edge[edge] (la.center);
		\draw (t1.center) edge[edge] (lb.center);
		\draw (la.center) edge[edge] (la-1);
		\draw (la.center) edge[edge] (la-2);
		\draw (lb.center) edge[edge] (lb-w);
		\draw (lb.center) edge[edge] (lb-b);
		
		\node (a2) at ($(t2)+(90:\innerLength)$) {};
		\node[vertexW] (a2-w) at ($(a2)+(90+\leafAngle:\leafLength)$) {};
		\node[vertexB] (a2-b) at ($(a2)+(90-\leafAngle:\leafLength)$) {};
		
		\draw (a2.center) edge[edge] (a2-w);
		\draw (a2.center) edge[edge] (a2-b);
		\draw (t2.center) edge[edge] (a2.center);
		
		\node (b3) at ($(t3)+(270:\innerLength)$) {};
		\node[vertexW] (b3-w) at ($(b3)+(270-\leafAngle:\leafLength)$) {};
		\node[vertexB] (b3-b) at ($(b3)+(270+\leafAngle:\leafLength)$) {};
		
		\draw (b3.center) edge[edge] (b3-w);
		\draw (b3.center) edge[edge] (b3-b);
		\draw (t3.center) edge[edge] (b3.center);
		
		\node (a4) at ($(t4)+(90:\innerLength)$) {};
		\node[vertexW] (a4-w) at ($(a4)+(90+\leafAngle:\leafLength)$) {};
		\node[vertexB] (a4-b) at ($(a4)+(90-\leafAngle:\leafLength)$) {};
		
		\draw (a4.center) edge[edge] (a4-w);
		\draw (a4.center) edge[edge] (a4-b);
		\draw (t4.center) edge[edge] (a4.center);
		
		\node (ra) at ($(t5)+(2*360/3-180:\innerLength)$) {};
		\node[vertexW] (ra-1) at ($(ra)+(2*360/3-180-\leafAngle:\leafLength)$) {};
		\node[vertexW] (ra-2) at ($(ra)+(2*360/3-180+\leafAngle:\leafLength)$) {};
		\node (rb) at ($(t5)+(360/3-180:\innerLength)$) {};
		\node[vertexB] (rb-w) at ($(rb)+(360/3-180-\leafAngle:\leafLength)$) {};
		\node[vertexW] (rb-b) at ($(rb)+(360/3-180+\leafAngle:\leafLength)$) {};
		
		\draw (t5.center) edge[edge] (ra.center);
		\draw (t5.center) edge[edge] (rb.center);
		\draw (ra.center) edge[edge] (ra-1);
		\draw (ra.center) edge[edge] (ra-2);
		\draw (rb.center) edge[edge] (rb-w);
		\draw (rb.center) edge[edge] (rb-b);
	\end{tikzpicture}
	\caption{The linear structure of a perfect matching decomposition of width $2$ with a claw on each side and two vertices from different colour classes added in each step.
	The filled vertices in the figure represent the leaves mapped to a vertex in $V_1$ and the empty vertices represent the leaves mapped to vertices in $V_2.$}
	\label{fig:brace-width-2-decomposition}
\end{figure}

\begin{corollary}
    \label{lemma:decomp_2_structure}
    Let $\Brace{T,\DecompBijection{}}$ be a perfect matching decomposition of width $2$ of a brace $B$ with $2n$ vertices, then
    \begin{itemize}
        \item $\Spine{\Spine{T}}$ is a path on $n-2$ vertices $t_1,\dots,t_{n-2}$
        \item both $\CutG{B}{t_1t_2}$ and $\CutG{B}{t_{n-1}t_{n-2}}$ have a shore that is a claw with centre vertices of different colour, and
        \item let $X_i\coloneqq\DecompBijection{T_i}$ where $T_i$ is the component of $T-t_it_{i+1}$ that contains $t_1$ for all $i\in\Set{1,\dots,n-3}$, then $X_i\setminus X_{i-1}$ contains exactly two vertices of different colour.
    \end{itemize}
\end{corollary}

Of specific interest to us is the directed corollary of \cref{prop:minority_closed_shores} for the case $k=2$.

\begin{corollary}\label{cor:eclosedminorities}
	Let $B$ be a brace of perfect matching width two, $\Brace{T,\DecompBijection{}}$ be an optimal perfect matching decomposition of $G$, $e\in\Fkt{E}{\Spine{\Spine{T}}}$ and $X$ a shore of $\CutG{B}{e}$.
	Then no vertex of the minority of $X$ has a neighbour in $\Compl{X}$.
\end{corollary}

\subsection{Elimination Orderings}

So, given a perfect matching decomposition $\Brace{T,\DecompBijection{}}$ of width two for a brace $B$ we know that $\Spine{\Spine{T}}$ is a path and each of its endpoints can be identified with a claw in $B$.
Moreover, if the central vertex of such a claw is a vertex of $V_1$, then $\Spine{\Spine{T}}$ induces a linear ordering of $V_1$ which is uniquely determined by $\Brace{T,\DecompBijection{}}$ except for the order of the last three vertices.
Let $a\in V_1$ be any vertex in $V_1$ and $X_{a}\subseteq V_1$ be the set of vertices smaller or equal to $a$ in the ordering induced by $\Brace{T,\DecompBijection{}}$, then \cref{cor:eclosedminorities} together with \cref{lemma:imbalance} implies $\Abs{X_a}+2=\Abs{\NeighboursG{B}{X_a}}$.
Inspired by this observation, we present a definition for elimination orderings in bipartite matching covered graphs.

\begin{definition}[Matching Elimination Width]
	Let $B$ be a bipartite matching covered graph and  $\Fkt{\Lambda}{V_i}$ be the set of all linear orderings of $V_i$ for $i\in\Set{1,2}$.
	Let $\lambda\in\Fkt{\Lambda}{V_i}$.
	For every $v\in V_i$ we define the set of \emph{reachable} vertices in $V_{3-i}$ as
	\begin{align*}
		\Reach{B}{\lambda}{v}&\coloneqq\NeighboursG{B}{\Prec{B}{\lambda}{v}}\text{, where}\\
		\Prec{B}{\lambda}{v}&\coloneqq\CondSet{v'\in V_i}{\Fkt{\lambda}{v'}\leq\Fkt{\lambda}{v}}.
	\end{align*}
	We also call these the \emph{reachability-set} and the \emph{predecessor-set} respectively.
	The \emph{width} of such an ordering is given by
	\begin{equation*}
		\Width{\lambda}\coloneqq\max_{v\in V_i} \Brace{\Abs{\Reach{B}{\lambda}{v}}-\Abs{\Prec{B}{\lambda}{v}}}.
	\end{equation*}
	Now the \emph{matching elimination width} of $B$ (with respect to $V_i$) is defined as
	\begin{equation*}
		\MEOW{i}{B}\coloneqq \min_{\lambda\in\Lambda\Brace{V_i}}\Width{\lambda}.
	\end{equation*}
\end{definition}

Please note that by \cref{thm:bipartiteextendability} $\Abs{\Reach{B}{\lambda}{v}}-\Abs{\Prec{B}{\lambda}{v}}\geq 0$ for all $\lambda\in\Fkt{\Lambda}{V_i}$ and all $v\in V_i$.
Moreover, if $v$ is not the largest vertex of $\lambda$, then $\Abs{\Reach{B}{\lambda}{v}}-\Abs{\Prec{B}{\lambda}{v}}\geq 1$ as $B$ is matching covered.
Also note that, in case $\lambda$ is an ordering whose width with respect to $V_1$ is some value $k$, then the ordering $\lambda'$ obtained by ordering the vertices of $V_2$ according to their appearance as neighbours of the vertices of $V_1$, and then reversing this order, is of the same width as $\lambda$, but now with respect to $V_2$.
Hence the choice of $i\in\Set{1,2}$ does not influence the width.

What follows is a characterisation of braces of perfect matching width two in terms of their matching elimination width.
To be more precise, we show that an ordering of the vertices in $V_1$ of width two can be used to construct a perfect matching decomposition $\Brace{T,\DecompBijection{}}$ of width two such that $\Spine{\Spine{T}}$ is a path.
Also, any linear ordering of $V_1$ obtained from such a path in a perfect matching decomposition $\Brace{T,\DecompBijection{}}$ of width two provides an ordering of $V_1$ of width two.

\begin{theorem}\label{thm:width2elimination}
	Let $B$ be a brace on at least $6$ vertices.
	Then $\pmw{B}=2$ if and only if $\MEOW{1}{B}=2$.
\end{theorem}

\begin{proof}
	First, let $\Brace{T,\DecompBijection{}}$ be a perfect matching decomposition for $B$ of width two.
	Then, by \cref{lemma:cubicwidth2trees}, $\Spine{T}$ is cubic and by \cref{prop:linearwidth2trees}, $\Spine{\Spine{T}}$ is a path.
	Let $n \coloneqq \Abs{V_1}$, then $\Abs{\V{B}}=2n$ and $T$ has $2n$ leaves.
	So by \cref{lemma:cubictrees}, $\Spine{T}$ has $2n-2$ vertices and as $\Spine{T}$ has a leaf for every two vertices of $B$, $\Abs{\Leaves{\Spine{T}}}=n$.
	
	Thus, $\Spine{\Spine{T}}$ has $n-2$ vertices, let $t_1,\dots,t_{n-2}$ be its vertices ordered by occurrence and $t_1$ being the endpoint that, by \cref{lemma:decomp_2_structure}, corresponds to a claw in $B$ whose central vertex is $v_1 \in V_1$.
	We define a bijective function $\lambda^{-1}\colon\Set{1,\dots,n}\rightarrow V_1$ whose inverse provides the desired ordering.
	We set $\Fkt{\lambda^{-1}}{1}\coloneqq v_1$.
	
	For each $i\in\Set{1,\dots,n-3}$ let $X_i\coloneqq\DecompBijection{T_i}$ where $T_i$ is the component of $T-t_it_{i+1}$ that contains $t_1$.
	By our definition of $v_1$ and $t_1$, $X_1\cap V_1=\Set{v_1}$.
	Now, consider $i\in\Set{2,\dots,n-3}$.
	Clearly $X_j\subseteq X_i$ for all $j<i$ and by \cref{lemma:decomp_2_structure}, $X_i\setminus X_{i-1}$ contains exactly two vertices, one being $u_i \in V_2$ and the other one being $v_i \in V_1$.
	Set $\Fkt{\lambda^{-1}}{i}\coloneqq v_i$.
	At last let $\Set{v_{n-2},v_{n-1},v_n}=\Compl{X_{n-3}}\cap V_1$ where the order of these three vertices is chosen arbitrarily and set $\Fkt{\lambda^{-1}}{j}\coloneqq v_j$ for all $j\in\Set{n-2,n-1,n}$.
	
	Now, $\lambda=\Brace{\lambda^{-1}}^{-1}$ is a linear ordering of $V_1$.
	Note that $\MEOW{1}{G}\geq 2$ due to \cref{thm:bipartiteextendability}.
	Hence it is only left to show that $\Width{\lambda}=2$.
	
	Let $v\in V_1$ be chosen arbitrarily.
	If $v\in\Set{v_{n-2},v_{n-1},v_n}$ we have nothing to show, so suppose $v=v_i$ for some $i\in\Set{1,\dots,n-3}$.
	Then, $\Reach{B}{\lambda}{v}=X_i\cap V_2$ and $\Prec{B}{\lambda}{v}=X_i\cap V_1=\Set{v_1,\dots,v_i}$.
	\Cref{lemma:imbalance} yields $\Imbalance{X_i}=2$ and as $\Set{v_1}$ is the minority of $X_1$, we obtain that $V_1$ contains the minority of $X_i$ from \cref{lemma:decomp_2_structure}.
	Therefore, $\Abs{\Reach{B}{\lambda}{v}-\Prec{B}{\lambda}{v}}=2$.
	As $i$ was chosen arbitrarily, $\Width{\lambda}=2$ and thus $\MEOW{1}{G}=2$.
	
	Second, for the reverse direction, let $\lambda$ be a linear ordering of $V_1$ of width two, and let $n\coloneqq\Abs{V_1}$.
	Since $B$ is a brace, $\Abs{\Reach{B}{\lambda}{v}}-\Abs{\Prec{B}{\lambda}{v}}\geq 2$ for all $v\in V_1$ with $\Fkt{\lambda}{v} \leq n-2$.
	Let $X_1\coloneqq\Set{v_1}\cup \NeighboursG{B}{v_1}$ and for all $i\in\Set{1,\dots,n-3}$ let $X_i\coloneqq X_{i-1}\cup\Set{v_i}\cup\NeighboursG{B}{v_i}$ and then let $X_{n-2}\coloneqq X_{n-3}\cup\Set{v_{n-2},v_{n-1},v_n}\cup\NeighboursG{B}{\Set{v_{n-2},v_{n-1},v_n}}$.
	We claim that $\MatPor{\CutG{B}{X_i}}=2$ for all $i\in\Set{1,\dots,n-2}$ and $\Abs{X_j}-\Abs{X_{j-1}}=2$ for all $j\in\Set{2,\dots,n-2}$ as well as $\Abs{X_1}=\Abs{X_{n-2}\setminus X_{n-3}}=4$.
	
	By construction, for all $i\in\Set{1,\dots,n-2}$, $\NeighboursG{B}{V_1\cap X_i}\subseteq X_i$ and thus $\MatPor{\CutG{B}{X_i}}=\Abs{X_i}-2\Abs{V_1\cap X_i}=\Abs{V_2\cap X_i}-\Abs{V_1\cap X_i}=2$, where the last equality follows from the width of $\lambda$.
	Now, consider $j\in\Set{1,\dots,n-3}$.
	By definition, $\Abs{X_j\cap V_1}-\Abs{X_{j-1}\cap V_1}=1$ and, as we have seen above, $\Abs{V_2\cap X_j}-\Abs{V_1\cap X_j}=\Abs{V_2\cap X_{j-1}}-\Abs{V_1\cap X_{j-1}}$ hence, $\Abs{X_j\cap V_2}-\Abs{X_{j-1}\cap V_2}=1$ as well.
	At last, clearly $\Abs{X_1}=4$ by definition and the width of $\lambda$.
	Moreover $\Abs{V_2\cap X_j}-\Abs{V_1\cap X_j}=2$ and thus $\Abs{X_{n-3}\cap V_2}-\Abs{X_{n-3}\cap V_1}=2$ implying $\Abs{X_{n-3}\cap V_2}=n-1$, so $\Abs{\Compl{X_{n-2}}}=4$.
	
	We now use the $X_i$ to construct a perfect matching decomposition of width two for $B$.
	The idea is simple, we introduce a path on $n-2$ vertices $t_1,\dots,t_{n-2}$ and identify $X_i$ with $t_i$ for all $i$.
	We construct a tree $T$ by first, introducing two new leaf neighbours for $t_1$ and $t_{n-2}$ and one new leaf neighbour for each $t_j$ with $j\in\Set{2,\dots,n-3}$ and second, introducing two leaf neighbours again for every leaf added in the first step.
	This results in the two endpoints of our original path being identified with four new leaves each, while every internal vertex of the path is identified with two leaves of the new tree $T$.
	We start creating $\DecompBijection{}$ by mapping the four leaves identified with $t_1$ to the vertices of $X_1$ and the four leaves identified with $t_{n-2}$ to the vertices of $\Compl{X_{n-3}}$.
	By our observations above, for each $j\in\Set{2,\dots,n-3}$, $\Abs{X_j}-\Abs{X_{j-1}}=2$ and so for each such $j$ we can map the two leaves of $T$ identified with $t_j$ to the two vertices in $X_j\setminus X_{j-1}$.
	The result is a perfect matching decomposition $\Brace{T,\DecompBijection{}}$ of $B$ and, since $\MatPor{\CutG{B}{X_i}}=2$ for all $i\in\Set{1,\dots,n-2}$, it is of width two.
	This completes our proof.
\end{proof}

\subsection{Edge-Maximal Braces of Perfect Matching Width Two}

Let $B$ be a brace of perfect matching width two and $\lambda$ a linear ordering of $V_1$ such that $\Width{\lambda}=2$.
Suppose for some $v\in V_1$ there is a $u\in\Reach{B}{\lambda}{v}$ with $uv\notin\E{B}$, then $\lambda$ is also a width-$2$-ordering of $B+uv$.
Using this observation, we can add edges to our brace until we reach a brace $B'$ such that $\MEOW{1}{B'+uv}>\MEOW{1}{B'}=2$ for every edge $uv$ with $v\in V_1$, $u\in V_2$ and $uv\notin\E{B'}$.

By following this idea of constructing an edge-maximal brace of perfect matching width two we obtain a special kind of bipartite graphs.
We call a brace $L_n=B$ a \emph{bipartite ladder} of \emph{order $n$} if $V_1=\Set{v_1,\dots,v_n}$, $V_2=\Set{u_1,\dots,u_n}$ and $\E{B}=E_1\cup E_2\cup E_3$ where
\begin{enumerate}
	
	\item $E_1\coloneqq\CondSet{v_iu_j}{\text{for all}~1\leq j\leq i\leq n}$,
	
	\item $E_2\coloneqq\CondSet{v_iu_{i+1}}{\text{for all}~1\leq i\leq n-1}$, and
	
	\item $E_3\coloneqq\CondSet{v_iu_{i+2}}{\text{for all}~1\leq i\leq n-2}$.
	
\end{enumerate}
The graphs $L_1$, which is a single edge, and $L_2$, which is isomorphic to $C_4$, are not too interesting as they are very small.
For $n\geq 3$ however these graphs grow more complex, see \cref{fig:ladders} for an illustration of $L_3$, $L_4$ and $L_5$.

\begin{figure}[!ht]
	\begin{center}
		\begin{tikzpicture}[scale=0.9]
			
			\pgfdeclarelayer{background}
			\pgfdeclarelayer{foreground}
			
			\pgfsetlayers{background,main,foreground}
			
			
			\begin{pgfonlayer}{main}
				
				\node (C) [] {};
				
				\node (C1) [v:ghost, position=90:33mm from C] {};
				
				\node (C2) [v:ghost, position=0:0mm from C] {};
				
				\node (C3) [v:ghost, position=270:33mm from C] {};

				
				
				\node(C1b1) [v:mainempty,position=0:0mm from C1] {};
				\node(C1b2) [v:mainempty,position=0:15mm from C1b1] {};
				\node(C1b3) [v:mainempty,position=0:15mm from C1b2] {};
				
				\node(C1a1) [v:main,position=270:17mm from C1b1] {};
				\node(C1a2) [v:main,position=0:15mm from C1a1] {};
				\node(C1a3) [v:main,position=0:15mm from C1a2] {};

				\node (L1) [v:ghost,position=180:15mm from C1b1] {$L_3=K_{3,3}$};
				
				\node (L1b1) [v:ghost,position=90:4mm from C1b1] {$u_1$};
				\node (L1b2) [v:ghost,position=90:4mm from C1b2] {$u_2$};
				\node (L1b3) [v:ghost,position=90:4mm from C1b3] {$u_3$};
				
				\node (L1a1) [v:ghost,position=270:4mm from C1a1] {$v_1$};
				\node (L1a2) [v:ghost,position=270:4mm from C1a2] {$v_2$};
				\node (L1a3) [v:ghost,position=270:4mm from C1a3] {$v_3$};
				
				
				
				\node(C2b1) [v:mainempty,position=0:0mm from C2] {};
				\node(C2b2) [v:mainempty,position=0:15mm from C2b1] {};
				\node(C2b3) [v:mainempty,position=0:15mm from C2b2] {};
				\node(C2b4) [v:mainempty,position=0:15mm from C2b3] {};
				
				\node(C2a1) [v:main,position=270:17mm from C2b1] {};
				\node(C2a2) [v:main,position=0:15mm from C2a1] {};
				\node(C2a3) [v:main,position=0:15mm from C2a2] {};
				\node(C2a4) [v:main,position=0:15mm from C2a3] {};

				\node (L2) [v:ghost,position=180:10mm from C2b1] {$L_4$};
				
				\node (L2b1) [v:ghost,position=90:4mm from C2b1] {$u_1$};
				\node (L2b2) [v:ghost,position=90:4mm from C2b2] {$u_2$};
				\node (L2b3) [v:ghost,position=90:4mm from C2b3] {$u_3$};
				\node (L2b4) [v:ghost,position=90:4mm from C2b4] {$u_4$};
				
				\node (L2a1) [v:ghost,position=270:4mm from C2a1] {$v_1$};
				\node (L2a2) [v:ghost,position=270:4mm from C2a2] {$v_2$};
				\node (L2a3) [v:ghost,position=270:4mm from C2a3] {$v_3$};
				\node (L2a4) [v:ghost,position=270:4mm from C2a4] {$v_4$};
				
				
				
				\node(C3b1) [v:mainempty,position=0:0mm from C3] {};
				\node(C3b2) [v:mainempty,position=0:15mm from C3b1] {};
				\node(C3b3) [v:mainempty,position=0:15mm from C3b2] {};
				\node(C3b4) [v:mainempty,position=0:15mm from C3b3] {};
				\node(C3b5) [v:mainempty,position=0:15mm from C3b4] {};
				
				\node(C3a1) [v:main,position=270:17mm from C3b1] {};
				\node(C3a2) [v:main,position=0:15mm from C3a1] {};
				\node(C3a3) [v:main,position=0:15mm from C3a2] {};
				\node(C3a4) [v:main,position=0:15mm from C3a3] {};
				\node(C3a5) [v:main,position=0:15mm from C3a4] {};

				\node (L3) [v:ghost,position=180:10mm from C3b1] {$L_5$};
				
				\node (L3b1) [v:ghost,position=90:4mm from C3b1] {$u_1$};
				\node (L3b2) [v:ghost,position=90:4mm from C3b2] {$u_2$};
				\node (L3b3) [v:ghost,position=90:4mm from C3b3] {$u_3$};
				\node (L3b4) [v:ghost,position=90:4mm from C3b4] {$u_4$};
				\node (L3b5) [v:ghost,position=90:4mm from C3b5] {$u_5$};
				
				\node (L3a1) [v:ghost,position=270:4mm from C3a1] {$v_1$};
				\node (L3a2) [v:ghost,position=270:4mm from C3a2] {$v_2$};
				\node (L3a3) [v:ghost,position=270:4mm from C3a3] {$v_3$};
				\node (L3a4) [v:ghost,position=270:4mm from C3a4] {$v_4$};
				\node (L3a5) [v:ghost,position=270:4mm from C3a5] {$v_5$};
				
				

				

				

				

				
				
			\end{pgfonlayer}
			
			
			\begin{pgfonlayer}{background}
				
				\draw (C1b1) [e:main] to (C1a1);
				\draw (C1b1) [e:main] to (C1a2);
				\draw (C1b1) [e:main] to (C1a3);
				
				\draw (C1b2) [e:main] to (C1a2);
				\draw (C1b2) [e:main] to (C1a3);
				
				\draw (C1b3) [e:main] to (C1a3);
				
				\draw (C1b2) [e:main,color=DarkGray] to (C1a1);
				\draw (C1b3) [e:main,color=Gray] to (C1a1);
				\draw (C1b3) [e:main,color=DarkGray] to (C1a2);

				\draw (C2b1) [e:main] to (C2a1);
				\draw (C2b1) [e:main] to (C2a2);
				\draw (C2b1) [e:main] to (C2a3);
				\draw (C2b1) [e:main] to (C2a4);
				
				\draw (C2b2) [e:main] to (C2a2);
				\draw (C2b2) [e:main] to (C2a3);
				\draw (C2b2) [e:main] to (C2a4);
				
				\draw (C2b3) [e:main] to (C2a3);
				\draw (C2b3) [e:main] to (C2a4);
				
				\draw (C2b4) [e:main] to (C2a4);
				
				\draw (C2b2) [e:main,color=DarkGray] to (C2a1);
				\draw (C2b3) [e:main,color=Gray] to (C2a1);
				\draw (C2b3) [e:main,color=DarkGray] to (C2a2);
				\draw (C2b4) [e:main,color=Gray] to (C2a2);
				\draw (C2b4) [e:main,color=DarkGray] to (C2a3);

				\draw (C3b1) [e:main] to (C3a1);
				\draw (C3b1) [e:main] to (C3a2);
				\draw (C3b1) [e:main] to (C3a3);
				\draw (C3b1) [e:main] to (C3a4);
				\draw (C3b1) [e:main] to (C3a5);
				
				\draw (C3b2) [e:main] to (C3a2);
				\draw (C3b2) [e:main] to (C3a3);
				\draw (C3b2) [e:main] to (C3a4);
				\draw (C3b2) [e:main] to (C3a5);
				
				\draw (C3b3) [e:main] to (C3a3);
				\draw (C3b3) [e:main] to (C3a4);
				\draw (C3b3) [e:main] to (C3a5);
				
				\draw (C3b4) [e:main] to (C3a4);
				\draw (C3b4) [e:main] to (C3a5);
				
				\draw (C3b5) [e:main] to (C3a5);
				
				\draw (C3b2) [e:main,color=DarkGray] to (C3a1);
				\draw (C3b3) [e:main,color=Gray] to (C3a1);
				\draw (C3b3) [e:main,color=DarkGray] to (C3a2);
				\draw (C3b4) [e:main,color=Gray] to (C3a2);
				\draw (C3b4) [e:main,color=DarkGray] to (C3a3);
				\draw (C3b5) [e:main,color=Gray] to (C3a3);
				\draw (C3b5) [e:main,color=DarkGray] to (C3a4);

			\end{pgfonlayer}	
			
			\begin{pgfonlayer}{foreground}

			\end{pgfonlayer}
		\end{tikzpicture}
	\end{center}
	\caption{The bipartite ladders of order $3$, $4$, and $5$.
    Edges of $E_1$ are black, the ones from $E_2$ are dark grey and the edges from $E_3$ are light grey.}
	\label{fig:ladders}
\end{figure}

The following corollary is a consequence of \cref{thm:bipartiteextendability}.

\begin{corollary}\label{cor:addingedges}
	Let $B$ be a brace and $v_1\in V_1$, $v_2\in V_2$ such that $v_1v_2\notin \E{B}$.
	Then $B+v_1v_2$ is a brace.
\end{corollary}

This corollary allows the construction of edge-maximal braces of width two we are aiming for.
We conclude this section with a second characterisation of perfect matching width two braces, this time in terms of edge-maximal supergraphs.

\begin{theorem}\label{thm:ladders}
	Let $B$ be brace with $\Abs{V_1}=n$.
	Then, $\pmw{B}=2$ if and only if $B\subseteq L_n$.
\end{theorem}

\begin{proof}
	We start by proving that every conformal subgraph of $L_n$ is of perfect matching width $2$ or isomorphic to $K_2$.
	To do so, by \cref{lemma:subgraphwidth}, it suffices to show $\pmw{L_n}=2$ for all $n\in\N$ with $n\geq 2$.
	The definition of $L_n$ directly provides an ordering $\lambda$ of $V_1=\Set{v_1,\dots,v_n}$ with $\Fkt{\lambda}{v_i}=i$.
	We prove that $\Width{\lambda}=2$.
	Let $i\in\Set{1,\dots,n-3}$ be arbitrary.
	By definition, $\NeighboursG{L_n}{v_i}=\Set{u_1,\dots,u_{i+2}}\subseteq\Reach{L_n}{\lambda}{v_i}$.
	Moreover, as $\NeighboursG{L_n}{v_j}\subseteq\NeighboursG{L_n}{v_i}$ for all $j\leq i$, $\NeighboursG{L_n}{v_i}=\Reach{L_n}{\lambda}{v_i}$.
	Therefore, $\Abs{\Reach{L_n}{\lambda}{v_i}}-\Abs{\Prec{L_n}{\lambda}{v_i}}=2$ for all $i\in\Set{1,\dots,n-2}$ and thus, $\Width{\lambda}=2$.
	By \cref{thm:width2elimination}, this implies $\pmw{L_n}=2$, as desired.
	
	Now, let $B$ be a brace of perfect matching width two.
	Then, there is an ordering $\lambda$ of $V_1$ of width two by \cref{thm:width2elimination}.
	Let us number the vertices of $V_1$ according to $\lambda$, so for all $i\in\Set{1,\dots,n}$ let $v_i\coloneqq\Fkt{\lambda^{-1}}{i}$.
	We construct a numbering of the vertices in $V_2$ as follows.
	Let $\NeighboursG{B}{v_1}=\Set{u_1,u_2,u_3}$ be numbered arbitrarily.
	The size of the neighbourhood of $a_1$ follows immediately from the width of $\lambda$ and the fact that $B$ is a brace.
	Now, as a consequence of \cref{lemma:decomp_2_structure}, for every $i\in\Set{1,\dots,n-2}$, $\Reach{B}{\lambda}{v_i}\setminus\Reach{B}{\lambda}{v_{i-1}}$ contains exactly one vertex, which is in $V_2$.
	Let $u_{i+2}$ be this vertex.
	Now, $\NeighboursG{B}{v_i}\subseteq\Reach{B}{\lambda}{v_i}$ for all $i\in\Set{1,\dots,n}$ and thus $B$ does not contain an edge that does not obey the definition of a bipartite ladder with respect to the orderings of $V_1$ and $V_2$ as obtained above.
	If there are indices $i\in\Set{1,\dots,n}$ and $j\in\Set{1,\dots,n}$ such that $v_iu_j\notin\E{B}$, but $j\leq i+2$, then we simply add the edge $v_iu_j$ to $B$.
	By \cref{cor:addingedges} $B+v_iu_j$ is still a brace and by choice of $i$ and $j$, adding this edge does not change the predecessor- and reachability-sets of any vertices in $V_1$, hence $\lambda$ is still an ordering of width two for $G+v_iu_j$.
	Thus, we can keep adding edges in this fashion until we do not find such a pair of indices any more.
	In that case let $B'$ be the newly obtained brace.
	By construction, $B'$ is isomorphic to $L_n$ and thus $B$ is a conformal subgraph of $L_n$.
\end{proof}

\section{Perfect Matching Width and Treewidth}\label{sec:comparissontotreewidth}

A natural question for any new width parameter is how it compares to other, already known parameters.
We have already seen a way to relate the perfect matching width of bipartite graphs and directed treewidth.
However, to apply our findings the graph itself has to be transformed.
In the first part of this short section we discuss the relation between the (undirected) treewidth\footnote{The definition of treewidth can be found in \cite{vatshelle2012new} by the interested reader.} of $G$ and its perfect matching width.
To do this we use a parameter introduced by Vatshelle \cite{vatshelle2012new} which is already known to be equivalent to treewidth but is much closer to perfect matching width in spirit.

Let $G$ be a graph and $X\subseteq\V{G}$.
We denote by $\Fkt{\MM}{\CutG{G}{X}}$ the number $\MatNum{\InducedSubgraph{G}{\CutG{G}{X}}}$ which is the maximum number of pairwise disjoint edges in $\CutG{G}{X}$.

A \emph{maximum matching decomposition} of $G$ is a tuple $\Brace{T,\DecompBijection{}}$ where $T$ is a cubic tree and $\DecompBijection{}\colon\Leaves{T}\rightarrow\V{G}$ is a bijection.

Recall from the definition of perfect matching width the associated edge cut $\CutG{G}{e}$ for every edge $e\in\E{T}$.
The \emph{width} of a maximum matching decomposition $\Brace{T,\DecompBijection{}}$ is defined as the maximum $\Fkt{\MM}{\CutG{G}{e}}$ over all $e\in\E{T}$, and the \emph{maximum matching width} of $G$, denoted by $\mmw{G}$, is the minimum width over all maximum matching decompositions of $G$.

\begin{theorem}[\cite{vatshelle2012new,jeong2018maximum}]\label{thm:mmw}
	Let $G$ be a graph and let $\tw{G}$ denote the treewidth of $G.$
	Then $\mmw{G}\leq\tw{G}+1\leq 3\mmw{G}.$
\end{theorem}

With this it is straight forward to bound the perfect matching width of a graph $G$ with a perfect matching in terms of its treewidth.

\begin{proposition}\label{prop:pmw;eqtw}
	Let $G$ be a graph with a perfect matching.
	Then $\pmw{G}\leq \tw{G}+1$.
\end{proposition}

\begin{proof}
	By \cref{thm:mmw} we have $\mmw{G}\leq \tw{G}+1$, so there exists a maximum matching decomposition $\Brace{T,\DecompBijection{}}$ for $G$ of width at most $\tw{G}+1$.
	Now let $M$ be any perfect matching of $G$ and $t_1t_2\in\E{T}$.
	Note that $M\cap\CutG{G}{\DecompBijection{T_{t_1}}}$ is a matching, hence $\Abs{M\cap\CutG{G}{t_1t_2}}\leq \Fkt{\MM}{\CutG{G}{t_1t_2}}\leq\tw{G}+1$.
	Indeed, as $M$ was chosen arbitrarily we have $\MatPor{\CutG{G}{t_1t_2}}\leq\tw{G}+1$ and thus $\Brace{T,\DecompBijection{}}$ is a perfect matching decomposition of $G$ of width at at most $\tw{G}+1$ and our claim follows.
\end{proof}

While treewidth gives us an upper bound on the perfect matching width of $G$, the reverse is not true in general.
With these findings we close this chapter.

\begin{proposition}\label{prop:pmwneqtw}
	For every $k\in\N$ with $k\geq 2$ there exists a brace $B_k$ with $\pmw{B_k}=2$ and $\tw{B_k}\geq k$. 
\end{proposition}

\begin{proof}
	First note that for every $t\in\N$, $\tw{K_{t+1}}=t$.
	So if we can show that $B_k$ contains $K_{k+1}$ as a minor we have proven $\tw{B_k}\geq k$ since treewidth is monotone under taking minors.
	Let $B_k$ be the bipartite ladder $L_{k+1}$ of order $k+1$.
	Then, by \cref{thm:ladders}, $\pmw{B_k}=2$ for all $k$.
	Now let $k+1\geq 2$ be chosen arbitrarily.
	We choose the perfect matching $M \coloneqq \CondSet{u_iv_i}{i\in\Set{1,\dots,k+1}}$.
	By definition of $L_{k+1}$ we know $u_iv_j$ for every $i\in\Set{1,\dots,k+1}$ and every $j\in\Set{i,\dots,k+1}$, so by contracting all edges in $M$ we obtain a graph on $k+1$ vertices with $\Choose{k+1}{2}$ edges.
	So $B_k/M$ is isomorphic to $K_{k+1}$ and we are done.
\end{proof}

\section{Computing Perfect Matching Decompositions of Width Two}\label{sec:algorithm}
\label{sec:alg}

In this section we provide an explicit polynomial time algorithm to compute an optimal perfect matching decomposition for braces of perfect matching width 2.
We do so by first finding a matching elimination ordering, as is possible due to \cref{thm:width2elimination}.
By \cref{lemma:decomp_2_structure}, such a construction has to start in a degree-3 vertex that builds a claw together with its neighbours.
This in particular allows us to immediately discard any claw-free graph.
Then the construction proceeds by choosing in each step a new vertex from the same colour class whose neighbourhood contains at most one vertex that is not already in the neighbourhood of the previously chosen vertices.
In the correct decomposition such a vertex exists by \cref{cor:eclosedminorities} and \cref{lemma:imbalance}.
Thus, if at some point there is no such vertex to pick there are two possible reasons.
Either the initial claw was not the optimal one to choose, then the algorithm starts over with a different claw.
Or, the brace has perfect matching width at least 3, this the algorithm concludes when having tried all possible claws.

	\begin{algorithm}[ht]
		\caption{Compute width-$2$-ordering}\label{alg:order}
		\begin{algorithmic}[1]
			\Procedure{order}{$V_2$, $V_1$}
			\State $\lambda^{-1}\gets\emptyset$
			\ForAll{$a\in V_1$}\label{step:chooseclaw}
				\State{$\lambda^{-1}\gets\emptyset$}
				\If{$\Abs{\Neighbours{a}}=3$}
					\State $\Fkt{\lambda^{-1}}{1}\gets a$
					\State $U\gets V_1\setminus\Set{a}$
					\State $P\gets \Set{a}$
					\ForAll{$i\in\Set{2,\dots,\Abs{V_1}}$}\label{step:choosea}
						\ForAll{$a'\in U$}\label{step:chooseap}
							\If{$\Abs{\Neighbours{a'}\setminus\Neighbours{P}}\leq 1$}\label{step:compareneighbourhood}
								\State $\Fkt{\lambda^{-1}}{i}\gets a'$
								\State $P\gets P\cup\Set{a'}$
								\State $U \gets U\setminus\Set{a'}$
								\State \textbf{break}
							\EndIf
						\EndFor
						\If{$\Fkt{\lambda^{-1}}{i}=\emptyset$}
							\State{\textbf{break}}
						\EndIf
					\EndFor
					\If{$\Fkt{\lambda^{-1}}{\Abs{V_1}}\neq\emptyset$}
						\textbf{return} $\lambda$
					\EndIf
				\EndIf
			
			\EndFor
			\State \textbf{return} $B$ is not of perfect matching width $2$.
			\EndProcedure
		\end{algorithmic}
	\end{algorithm}
	
	\begin{lemma}\label{lemma:optorder}
	Let $B=\Brace{V_1\cup V_2,E}$ be a brace.
	Then \cref{alg:order} computes an ordering $\lambda$ of width $2$ on input $V_2$ and $V_1$ if and only if $\pmw{B}=2$.
\end{lemma}

	\begin{proof}
		First, assume \cref{alg:order} returns an ordering $\lambda$ for the input $V_2$ and $V_1$.
		Then, we can consider the sets $\Prec{B}{\lambda}{a}$ and $\Reach{B}{\lambda}{a}$.
		
		Proving $\Abs{\Reach{B}{\lambda}{\Fkt{\lambda^{-1}}{j}}}-\Abs{\Prec{B}{\lambda}{\Fkt{\lambda^{-1}}{j}}}\leq 2$ for all $j\in\Set{1,\dots,\Abs{V_1}}$ by induction shows $\Width{\lambda}=2$ as $2\leq \Width{\lambda}$ since $B$ is a brace.
		If $j\in\Set{1,\Abs{V_1}-1,\Abs{V_1}}$, there is nothing to show.
		So, suppose $2\leq j\leq\Abs{V_1}-3$ and let $a\coloneqq\Fkt{\lambda^{-1}}{j}$.
		That is, $a$ is chosen in the iteration for $i=j$ in \cref{step:choosea}.
		Let $P_a$ and $U_a$ be the sets $P$ and $U$ during this step of the algorithm.
		The set $P_a$ contains all vertices that were previously chosen by \cref{alg:order} and thus are smaller than $a$ with respect to $\lambda$.
		Hence $\Prec{B}{\lambda}{a}= P_a\cup\Set{a}$ and $\Prec{B}{\lambda}{\Fkt{\lambda^{-1}}{j-1}}=P_a$.
		With $a$ being chosen at step $j$, we know $\Abs{\Neighbours{a}\setminus\Neighbours{P_a}}\leq 1$.
		Therefore,
		\begin{align*}
			\Abs{\Reach{B}{\lambda}{a}}-\Abs{\Prec{B}{\lambda}{a}}=&\Abs{\Neighbours{P_a\cup\Set{a}}}-\Abs{P_a\cup\Set{a}]}\\
			\leq&\Abs{\Neighbours{P_a}}+1-\Brace{\Abs{P_a}+1}\\
			\leq&\Abs{P_a}+3-\Brace{\Abs{P_a}+1}=2.
		\end{align*}
		Hence, by \cref{thm:width2elimination}, $\Width{\lambda}=2$ and therefore $\pmw{B}=2$.

		Second, assume $\pmw{B}=2$.
		By \cref{thm:width2elimination}, there exists an ordering $\sigma$ of $V_1$ with $\Width{\sigma}=2$.
		We have already seen that if \cref{alg:order} returns an ordering $\lambda$, it is of width $2$.
		So what remains to show is that the algorithm returns an ordering.
		Suppose it does not.
		
		Let $a_1 \coloneqq \Fkt{\lambda^{-1}}{1}$.
		Since \cref{alg:order} only terminates without returning an ordering when it looped through all elements for the choice in \cref{step:chooseclaw}, it reaches the point where it chooses $a_1$.
		Now, \cref{alg:order} can choose the next element in \cref{step:chooseap} fulfilling the demand in \cref{step:compareneighbourhood} according to the ordering $\lambda$.
		Since it does not end up returning an ordering it eventually differs from any optimal ordering and then reaches the point $2 \leq k \leq \Abs{V_1}$ at which there is no element to choose in \cref{step:chooseap} fulfilling the demand in \cref{step:compareneighbourhood}.
		Let $a_1 \dots,a_k$ be elements of $V_1$ that \cref{alg:order} ordered this way so far before it gets stuck.
		Let $\sigma$ be chosen among all width-$2$-orderings of $V_1$ maximising $j\in\Set{1,\dots,k-1}$ such that $\Fkt{\sigma^{-1}}{i}=a_i$ for all $1\leq i< j$ and $\Fkt{\sigma^{-1}}{j}\neq a_j$.
		We refer to the elements after $a_j$ in $\sigma$ by $y_h\coloneqq\Fkt{\sigma^{-1}}{h}$ for all $h\in\Set{j+1,\dots,\Abs{V_1}}$.
		By the definition of the algorithm, $\Abs{\Neighbours{a_j}\setminus\Neighbours{a_1,\dots,a_{j-1}}}\leq 1$.
		Let $\sigma'$ be the ordering obtained from $\sigma$ by inserting $a_j$ at the position $j$ instead of its position $j+x$ in $\sigma$.
		So $\sigma'$ contains the elements of $V_1$ in the order $a_1,\dots,a_j,y_{j+1},\dots,y_{j+x-1},y_{j+x+1},\dots,y_{\Abs{V_1}}$.	
		
		Suppose, $\Width{\sigma'}\geq 3$.
		There is a vertex $y_{h'}$ with $h'\in\Set{j+1,\dots,j+x-1}$ such that
		\begin{align*}
			\Abs{\Reach{B}{\sigma'}{y_{h'}}}-\Abs{\Prec{B}{\sigma'}{y_{h'}}}\geq 3.
		\end{align*}
		But $\Abs{\Prec{B}{\sigma'}{y_{h'}}}=\Abs{\Prec{B}{\sigma}{y_{h'}}}+1$ and with
		\begin{align*}
			\Abs{\Neighbours{a_j}\setminus\Neighbours{a_1,\dots,a_{j-1}}}\leq 1
		\end{align*}
		we obtain $\Abs{\Reach{B}{\sigma'}{y_{h'}}}\leq\Abs{\Reach{B}{\sigma}{y_{h'}}}+1$.
		Thus,
		\begin{align*}
			\Abs{\Reach{B}{\sigma}{y_{h'}}}-\Abs{\Prec{B}{\sigma}{y_{h'}}}\geq 3,
		\end{align*}
		which contradicts $\sigma$ to be of width $2$.
		Hence $\Width{\sigma'}=2$.
		However, this is a contradiction to the choice of $\sigma$ as $\sigma'$ now coincides on the first $j$ positions with the choice of \cref{alg:order}.
		Thus, the algorithm does not get stuck once it chose the right claw and therefore, \cref{alg:order} returns an ordering.
	\end{proof}

	So \cref{alg:order} produces an elimination ordering of width $2$ if and only if the brace $B$ that was given as input is of perfect matching width $2$.
	This ordering can be translated into a perfect matching decomposition of width $2$, as seen in the second part of the proof of \cref{thm:width2elimination}.
	Since all sets necessary for the construction of this decomposition can be computed from the ordering by iterating over edges and vertices of $B$ at most once, this procedure runs in polynomial time and thus, we obtain the following result which concludes this section.

	\begin{theorem}\label{thm:computewidth}
		Let $B=\Brace{V_1\cup V_2,E}$ be a brace.
		There is a polynomial time algorithm that computes a perfect matching decomposition of width $2$ if and only if $\pmw{B}=2$.
	\end{theorem}

\section{Bipartite Graphs of M-Perfect Matching Width Two}\label{sec:Mpmw}

\Cref{sec:pmw2} provides a complete characterisation of braces of perfect matching width two.
However, we are not able to lift this result to all bipartite matching covered graphs, since we do not know whether the braces of a matching covered bipartite graph of perfect matching width two are also of perfect matching width two themselves.
To be more precise, for a matching covered bipartite graph $B$ with $\pmw{B}=2$, the best we know about any brace $H$ of it is $\pmw{H}\in\Set{2,3,4}$ by \cref{cor:bricksandbraceslowerbound}.
We can however consider the $M$-perfect matching width instead since here \cref{lemma:Mpmwandmatminors} implies that $\Mpmw{M}{B}$ bounds $\Mpmw{\Remainder{M}{H}}{H}$.
Indeed, since $K_2$ is the only matching covered graph of $M$-perfect matching width one, $B$ has $M$-perfect matching width two if and only if every brace $H$ of $B$ has $\Remainder{M}{H}$-perfect matching width two.

In this section we present a full characterisation of the braces of $M$-perfect matching width two and thus, provide a description of all matching covered bipartite graphs that have a perfect matching $M$ such that their $M$-perfect matching width is $2$.

Key to this characterisation is the observation that, given a brace $B$, $2\leq\pmw{B}\leq \Mpmw{M}{B}$ for all $M\in\Perf{B}$.
So, if $\Mpmw{M}{B}=2$ for some $M$, then every optimal $M$-decomposition of $B$ also is an optimal perfect matching decomposition of $G$.
Therefore we can apply the results from \cref{sec:pmw2}.
This immediately implies a rather strict bound on the number of vertices, which in turn narrows down the braces of $M$-perfect matching width two to exactly two, namely $K_{3,3}$ and $C_4$.

\begin{proposition}\label{prop:mwidth2braces}
	Let $B$ be a brace, then the following statements are equivalent.
	\begin{enumerate}
		
		\item $\Mpmw{M}{B}=2$ for an $M\in\Perf{B}$,\label{prop1}
		
		\item $\Mpmw{M}{B}=2$ for all $M\in\Perf{B}$, and\label{prop2}
		
		\item $B$ is isomorphic to $C_4$ or $K_{3,3}$.\label{prop3}
		
	\end{enumerate}
\end{proposition}

\begin{proof}
	In order to prove this statement, we first deduce \cref{prop3} from \cref{prop1} and then observe that we can find the same type of decomposition for every $M\in\Perf{B}$ which then implies \cref{prop2}.
	
	Let $B$ be a brace and $M\in\Perf{B}$ such that $\Mpmw{M}{B}=2$, then $\pmw{B}=2$ as well.
	Let $\Brace{T,\DecompBijection{}}$ be an optimal $M$-decomposition for $B$, then it also is an optimal perfect matching decomposition of $B$.
	Now suppose $\Abs{\V{B}}\geq 8$.
	Then by \cref{lemma:decomp_2_structure}, there is an edge $e\in\E{\Spine{\Spine{T}}}$ such that $\CutG{B}{e}$ has a shore $X$ of size $4$ that induces a claw in $B$.
	In particular, $\Imbalance{X}=2$ and thus $X$ is not $M$-conformal.
	This is a contradiction to the definition of $M$-decompositions as $e$ is an inner edge of $T$.
	So $\Abs{\V{G}}\leq 6$.
	On at most $6$ vertices there are only two braces: $C_4$ and $K_{3,3}$.
	
	\begin{figure}[!ht]
		\begin{center}
			\begin{tikzpicture}[scale=0.8]
				
				\pgfdeclarelayer{background}
				\pgfdeclarelayer{foreground}
				
				\pgfsetlayers{background,main,foreground}

				\begin{pgfonlayer}{main}
					
					\node (C) [] {};
					
					\node (C1) [v:ghost, position=180:60mm from C] {};
					
					\node (C2) [v:ghost, position=0:0mm from C] {};
					
					\node (C3) [v:ghost, position=0:60mm from C] {};

					
					
					\node(C1a) [v:main,fill=black,position=135:16mm from C1] {};
					\node(C1b) [v:main,position=225:16mm from C1] {};
					\node(C1c) [v:main,fill=black,position=315:16mm from C1] {};
					\node(C1d) [v:main,position=45:16mm from C1] {};

					\node(L1a) [v:ghost,position=135:4.25mm from C1a] {$a$};
					\node(L1b) [v:ghost,position=225:4.25mm from C1b] {$b$};
					\node(L1c) [v:ghost,position=315:4.25mm from C1c] {$c$};
					\node(L1d) [v:ghost,position=45:4.25mm from C1d] {$d$};
					
					
					\node (block) [v:ghost,position=180:55mm from C1] {};
					
					
					
					\node(C2a) [v:tree,position=180:12mm from C2] {};
					\node(C2b) [v:tree,position=0:0mm from C2] {};
					\node(C2c) [v:tree,position=45:12mm from C2b] {};
					\node(C2d) [v:tree,position=225:12mm from C2a] {};
					\node(C2e) [v:tree,position=135:12mm from C2a] {};
					\node(C2f) [v:tree,position=315:12mm from C2b] {};
					
					\node (L2a) [v:ghost,position=135:4.25mm from C2e] {$a$};
					\node (L2b) [v:ghost,position=315:4.25mm from C2f] {$c$};
					\node (L2c) [v:ghost,position=45:4.25mm from C2c] {$b$};
					\node (L2d) [v:ghost,position=225:4.25mm from C2d] {$d$};

					

					

					

					

					
					
				\end{pgfonlayer}
				
				
				\begin{pgfonlayer}{background}
					
					\draw [e:coloredborder] (C1a) to (C1d);
					\draw [e:colored] (C1a) to (C1d);
					
					\draw [e:coloredborder] (C1b) to (C1c);
					\draw [e:colored] (C1b) to (C1c);			
					
					\draw (C1a) [e:main] to (C1b);
					\draw (C1c) [e:main] to (C1d);

					\draw (C2a) [e:main] to (C2b);
					\draw (C2a) [e:main] to (C2d);
					\draw (C2a) [e:main] to (C2e);
					\draw (C2b) [e:main] to (C2f);
					\draw (C2b) [e:main] to (C2c);
					
				\end{pgfonlayer}	
				
				\begin{pgfonlayer}{foreground}

				\end{pgfonlayer}
			\end{tikzpicture}
		\end{center}
		\caption{The brace $C_4$ together with a perfect matching $M$ and an $M$-decomposition $\Brace{T,\DecompBijection{}}$ of width two.}
		\label{fig:C4}
	\end{figure}
	
	First, consider $C_4$.
	Let $M\in\Perf{C_4}$ be a perfect matching.
	Then, $\V{C_4}=\Set{a,b,c,d}$ and without loss of generality $M=\Set{ad,bc}$.
	As $C_4$ is a cycle, the only other perfect matching of $C_4$ is $\E{C_4}\setminus M=\Set{ab,cd}$.
	We construct a perfect matching decomposition $\Brace{T,\DecompBijection{}}$ as follows.
	Take two vertices $t_1$ and $t_2$ joined by an edge.
	We create a cubic tree $T$ by adding two leaves $t_i^1$ and $t_i^2$ as new neighbours to each of the $t_i$ for $i\in\Set{1,2}$.
	Then, let $\DecompBijection{t_1^1}\coloneqq a$, $\DecompBijection{t_1^2}\coloneqq d$, $\DecompBijection{t_2^1}\coloneqq b$ and $\DecompBijection{t_2^2}\coloneqq c$ (see \cref{fig:C4}).
	Now, $\Brace{T,\DecompBijection{}}$ is an $M$-decomposition of $C_4$ and the matching porosity of every cut induced by an edge of $T$ is either one or two.
	Note that for the other perfect matching of $C_4$ we just have to adapt the mapping $\DecompBijection{}$ such that for each $i\in\Set{1,2}$ the leaves $t_i^1$ and $t_i^2$ are mapped to the endpoints of the same edge and thus $\Mpmw{M}{C_4}=2$ for all $M\in\Perf{C_4}$.
	
	\begin{figure}[!ht]
		\begin{center}
			\begin{tikzpicture}[scale=0.8]
				
				\pgfdeclarelayer{background}
				\pgfdeclarelayer{foreground}
				
				\pgfsetlayers{background,main,foreground}
				
				
				\begin{pgfonlayer}{main}
					
					\node (C) [] {};
					
					\node (C1) [v:ghost, position=180:60mm from C] {};
					
					\node (C2) [v:ghost, position=0:0mm from C] {};
					
					\node (C3) [v:ghost, position=0:60mm from C] {};

					
					
					\node(C1b) [v:main,position=90:8mm from C1] {};
					\node(C1a) [v:main,fill=black,position=180:14mm from C1b] {};
					\node(C1c) [v:main,fill=black,position=0:14mm from C1b] {};
					\node(C1e) [v:main,fill=black,position=270:8mm from C1] {};
					\node(C1d) [v:main,position=0:14mm from C1e] {};
					\node(C1f) [v:main,position=180:14mm from C1e] {};

					\node(L1a) [v:ghost,position=90:4.25mm from C1a] {$a$};
					\node(L1b) [v:ghost,position=90:4.25mm from C1b] {$b$};
					\node(L1c) [v:ghost,position=90:4.25mm from C1c] {$c$};
					\node(L1d) [v:ghost,position=270:4.25mm from C1d] {$d$};
					\node(L1e) [v:ghost,position=270:4.25mm from C1e] {$e$};
					\node(L1f) [v:ghost,position=270:4.25mm from C1f] {$f$};
					
					
					\node (block) [v:ghost,position=180:55mm from C1] {};
					
					
					
					\node(C2a) [v:tree,position=180:12mm from C2] {};
					\node(C2b) [v:tree,position=0:0mm from C2] {};
					\node(C2c) [v:tree,position=0:12mm from C2] {};
					\node(C2d) [v:tree,position=90:12mm from C2] {};
					\node(C2e) [v:tree,position=135:12mm from C2a] {};
					\node(C2f) [v:tree,position=225:12mm from C2a] {};
					\node(C2g) [v:tree,position=45:12mm from C2c] {};
					\node(C2h) [v:tree,position=315:12mm from C2c] {};
					\node(C2i) [v:tree,position=135:12mm from C2d] {};
					\node(C2j) [v:tree,position=45:12mm from C2d] {};
					
					\node (L2a) [v:ghost,position=135:4.25mm from C2e] {$a$};
					\node (L2f) [v:ghost,position=225:4.25mm from C2f] {$f$};
					\node (L2c) [v:ghost,position=45:4.25mm from C2g] {$c$};
					\node (L2d) [v:ghost,position=315:4.25mm from C2h] {$d$};
					\node (L2c) [v:ghost,position=135:4.25mm from C2i] {$b$};
					\node (L2d) [v:ghost,position=45:4.25mm from C2j] {$e$};

					

					

					

					

					
					
				\end{pgfonlayer}
				
				
				\begin{pgfonlayer}{background}
					
					\draw (C1a) [e:main] to (C1e);
					\draw (C1b) [e:main] to (C1f);
					\draw (C1c) [e:main] to (C1f);
					\draw (C1a) [e:main] to (C1d);
					\draw (C1b) [e:main] to (C1d);
					\draw (C1c) [e:main] to (C1e);
					
					\draw [e:coloredborder] (C1a) to (C1f);
					\draw [e:colored] (C1a) to (C1f);
					
					\draw [e:coloredborder] (C1b) to (C1e);
					\draw [e:colored] (C1b) to (C1e);
					
					\draw [e:coloredborder] (C1c) to (C1d);
					\draw [e:colored] (C1c) to (C1d);

					\draw (C2a) [e:main] to (C2b);
					\draw (C2a) [e:main] to (C2e);
					\draw (C2a) [e:main] to (C2f);
					\draw (C2b) [e:main] to (C2c);
					\draw (C2b) [e:main] to (C2d);
					\draw (C2c) [e:main] to (C2g);
					\draw (C2c) [e:main] to (C2h);
					\draw (C2d) [e:main] to (C2i);
					\draw (C2d) [e:main] to (C2j);
					
				\end{pgfonlayer}	
				
				\begin{pgfonlayer}{foreground}

				\end{pgfonlayer}
			\end{tikzpicture}
		\end{center}
		\caption{The brace $K_{3,3}$ together with a perfect matching $M$ and an $M$-decomposition $\Brace{T,\DecompBijection{}}$ of width two.}
		\label{fig:K33}
	\end{figure}
	
	Second consider $K_{3,3}$ and let $V_1=\Set{a,b,c}$ and $V_2=\Set{d,e,f}$ and $M=\Set{af,be,cd}$ a perfect matching of $K_{3,3}$.
	We again construct an $M$-decomposition $\Brace{T,\DecompBijection{}}$ of our brace.
	This time consider a claw on the vertices $\Set{t,t_1,t_2,t_3}$ such that $t$ is the central vertex.
	For each $i\in\Set{1,2,3}$ we introduce two new neighbours $t_i^1$ and $t_i^2$ to $t_i$ which are the leaves of our cubic tree $T$.
	Then let $\DecompBijection{t_1^1}\coloneqq a$ and $\DecompBijection{t_1^2}\coloneqq f$.
	For the remaining two edges of $M$ proceed analogously  by choosing an $i\in\Set{2,3}$ for each of the remaining edges and then mapping the leaves $t_i^1$ and $t_i^2$ to the endpoints of the chosen edge.
	Now, $\Brace{T,\DecompBijection{}}$ is an $M$-decomposition of $K_{3,3}$ and for every inner edge $e$ of $T$ the cut induced by $e$ has a shore of size two, hence $\Width{T,\DecompBijection{}}=2$ (see \cref{fig:K33} for an illustration).
	Again, we can adapt the same strategy for every perfect matching $M'\in\Perf{B}$ and thus $\Mpmw{M}{B}=2$ for all $M\in\Perf{B}$.
	
	We have seen that for each of the braces $C_4$ and $K_{3,3}$ the $M$-perfect matching width equals two for all perfect matchings $M$.
	So, in particular there exists such a matching and thus, \cref{prop2} implies \cref{prop1} again and the proof is complete.
\end{proof}

With \cref{prop:mwidth2braces} we are able to deduce a similar theorem for general bipartite matching covered graphs of $M$-perfect matching width two.

\begin{proposition}\label{thm:mwidth2graphs}
	Let $B$ be a bipartite matching covered graph, then the following statements are equivalent.
	\begin{enumerate}
		
		\item $\Mpmw{M}{B}=2$ for an $M\in\Perf{B}$,\label{Prop1}
		
		\item $\Mpmw{M}{B}=2$ for all $M\in\Perf{B}$, and\label{Prop2}
		
		\item Every brace of $B$ is either isomorphic to $C_4$ or to $K_{3,3}$.\label{Prop3}
		
	\end{enumerate}
\end{proposition}

\begin{proof}
    If $B$ is a brace, then the statement holds by \cref{prop:mwidth2braces}.
    Thus assume that $B$ contains a tight cut.
    
	By \cite{lovasz1987matching}, it suffices to show that if the statement holds for the two tight cut contractions $B_Z\coloneqq\ContractsTo{B}{Z}{v_Z}$, and $B_{\Compl{Z}}\coloneqq\ContractsTo{B}{\Compl{Z}}{v_{\Compl{Z}}}$ of a bipartite matching covered graph $B$ with tight cut $\CutG{B}{Z},$ then it also holds for $B$.
	
	By induction hypothesis, the three statements are equivalent for both $B_Z$ and $B_{\Compl{Z}}$.
	
	Assume $\Mpmw{M}{B}=2$ for an $M\in\Perf{B}$ (\cref{Prop1}), then by \cref{lemma:Mpmwandmatminors} 
	$\Mpmw{\Remainder{M}{B_Z}}{B_Z} = \Mpmw{\Remainder{M}{B_{\Compl{Z}}}}{B_{\Compl{Z}}}=2$ and thus, the braces of both $B_Z$ and $B_{\Compl{Z}}$ are isomorphic to $C_4$ or $K_{3,3}$.
	Since the braces of $B$ are exactly the union of the braces of $B_Z$ and $B_{\Compl{Z}}$, \cref{Prop3} holds for $B$ as well.
	
	Next, assume that \cref{Prop3} holds for $B$.
	Pick any matching $M'\in\Perf{B}$, then by induction hypothesis $\Mpmw{\Remainder{M'}{B_Z}}{B_Z} = \Mpmw{\Remainder{M'}{B_{\Compl{Z}}}}{B_{\Compl{Z}}}=2$.
	Let $e_Z\in \Remainder{M'}{B_Z}$ and $e_{\Compl{Z}}\in \Remainder{M'}{B_{\Compl{Z}}}$ be the two edges covering $v_Z$ and $v_{\Compl{Z}}$ in the respective contractions for the respective reductions of $M'$.
	Let $u_X$ be the endpoint of $e_X$ that is not $v_X$ for both $X\in\Set{Z,\Compl{Z}}$.
	Moreover, let $\Brace{T_X,\DecompBijection{}_X}$ be an optimal $\Remainder{M'}{X}$-decomposition of  $B_X$ for both $X\in\Set{Z,\Compl{Z}}$.
	In $T_Z$ there is a vertex $t_Z$ that is adjacent to the two leaves of $T_Z$ that are mapped to $v_Z$ and $u_Z$, let $t_{\Compl{Z}}$ be chosen analogously.
	Observe, that $M'=\Brace{\Brace{\Remainder{M'}{B_Z}\cup\Remainder{M'}{B_{\Compl{Z}}}}\setminus\Set{e_Z,e_{\Compl{Z}}}}\cup\Set{u_Zu_{\Compl{Z}}}$.
	We construct an $M'$-decomposition $\Brace{T',\DecompBijection{}'}$ as follows.
	Let $T'_X$ be obtained from $T_X$ be deleting the two leaves adjacent to $t_X$ for both $X\in\Set{Z,\Compl{Z}}$.
	Then, let $T''$ be the tree obtained from $T'_Z$ and $T'_{\Compl{Z}}$ by identifying $t_Z$ and $t_{\Compl{Z}}$, call the new vertex $t$.
	At last, let $T'$ be the tree obtained from $T''$ by adding a new vertex $t'$, the edge $tt'$ and two new leaves $t_1$ and $t_2$ adjacent to the new vertex $t'$.
	Then, $T'$ is a cubic tree and $\Abs{\Leaves{T'}}=\Abs{\V{B}}$.
	In the next step we define $\DecompBijection{}':\Leaves{T'}\rightarrow\V{B}$ as follows:
	\begin{equation*}
		\Fkt{\DecompBijection{}'}{\ell}\coloneqq\FourCases{\Fkt{\DecompBijection{}_Z}{\ell}}{\text{if}~\ell\in\Leaves{T_Z}\setminus\Set{\Fkt{\DecompBijection{}_Z^{-1}}{v_Z}},}{\Fkt{\DecompBijection{}_{\Compl{Z}}}{\ell}}{\text{if}~\ell\in\Leaves{T_{\Compl{Z}}}\setminus\Set{\Fkt{\DecompBijection{}_{\Compl{Z}}^{-1}}{v_{\Compl{Z}}}},}{u_{\Compl{Z}}}{\text{if}~\ell=t_1\text{, and}}{u_Z}{\text{if}~\ell=t_2.}
	\end{equation*}
	Now, $\Brace{T',\DecompBijection{}'}$ is an $M'$-decomposition of $B$.
	Moreover, let $e\in\E{T'}$ be an inner edge of $T'$, then either $e$ is an inner edge of $T_Z$ or $T_{\Compl{Z}}$ and by construction of $T'$ and the fact that $\CutG{B}{Z}$ is tight, $\MatPor{\CutG{B}{e}} \leq 2$, or $e=tt'$.
	In the later case, $\CutG{B}{e}$ has a shore of size two and thus $\MatPor{\CutG{B}{e}}=2$.
	Therefore, $\Width{T',\DecompBijection{}'}=2$ and so $M'$-$\pmw{B}=2$ for all $M'\in\Perf{B}$, that is \cref{Prop2} holds.
	Since \cref{Prop2} implies \cref{Prop1}, we are done.
\end{proof}

So, in order to recognise a bipartite matching covered graph $B$ of $M$-perfect matching width two, it suffices to check whether $B$ has a brace not isomorphic to $C_4$ or $K_{3,3}$.
Lov{\'a}sz has shown that the tight cut decomposition of a matching covered graph can be computed in polynomial time (see \cite{lovasz1987matching}) and thus, \cref{thm:mwidth2graphs} implies a polynomial recognition algorithm for bipartite matching covered graphs of $M$-perfect matching width two.
Moreover, the proof of \cref{thm:mwidth2graphs} is constructive and can be used to obtain an $M$-decomposition of width two for any $M\in\Perf{B}$, given a bipartite matching covered graph $B$ of $M$-perfect matching width two, from the decompositions of its braces.
As these braces are only $C_4$ and $K_{3,3}$, whose optimal $M$-decompositions are given in the proof of \cref{prop:mwidth2braces}, we obtain the following corollary.

\begin{corollary}
	Let $B$ be a bipartite matching covered graph and $M\in\Perf{B}$.
	Then, we can compute in polynomial time either an $M$-decomposition of width two, or a brace of $B$ that is neither isomorphic to $C_4$, nor to $K_{3,3}$.
\end{corollary}

In order to obtain \cref{thm:Mpmw2}, we consider the family of odd M\"obius ladders.

\begin{definition}[Odd M\"obius ladders]
	An \emph{odd M\"obius ladder} of \emph{order} $k\geq 1$ is the graph $\mathscr{M}_{4k+2}$ obtained from the cycle
	\begin{equation*}
        \Brace{v^1_0,v^2_0,v^1_1,v^2_1,v^1_2,\dots,v^2_{2k-1},v^1_{2k},v^2_{2k},v^1_0}
	\end{equation*}
	by adding the edges $v^1_iv^2_{k+i~\!\Brace{\bmod 2k+1}}$, for all $i\in\Set{0,\dots,2k}$.
	See \cref{fig:macuaigbraces} for an illustration.
\end{definition}

\begin{figure}[!ht]
	\begin{center}
		\begin{tikzpicture}[scale=0.7]
			
			\pgfdeclarelayer{background}
			\pgfdeclarelayer{foreground}
			
			\pgfsetlayers{background,main,foreground}
			
			\begin{pgfonlayer}{main}
				
				\node (Layer1) [v:ghost] {};
				
				\node (L1C1) [v:ghost,position=0:0mm from Layer1] {};

				\node (L1C1label) [v:ghost,position=270:19mm from L1C1] {$\mathscr{M}_{6}$};
				
				\node (1-1v1) [v:main,position=90:10mm from L1C1] {};
				\node (1-1v2) [v:mainempty,position=150:10mm from L1C1] {};
				\node (1-1v3) [v:main,position=210:10mm from L1C1] {};
				\node (1-1v4) [v:mainempty,position=270:10mm from L1C1] {};
				\node (1-1v5) [v:main,position=330:10mm from L1C1] {};
				\node (1-1v6) [v:mainempty,position=30:10mm from L1C1] {};
				
				\node (L1C2) [v:ghost,position=0:33mm from L1C1] {};
				
				\node (L1C2label) [v:ghost,position=270:19mm from L1C2] {$\mathscr{M}_{10}$};
				
				\node (1-2v1) [v:main,position=0:11.5mm from L1C2] {};
				\node (1-2v2) [v:mainempty,position=36:11.5mm from L1C2] {};
				\node (1-2v3) [v:main,position=72:11.5mm from L1C2] {};
				\node (1-2v4) [v:mainempty,position=108:11.5mm from L1C2] {};
				\node (1-2v5) [v:main,position=144:11.5mm from L1C2] {};
				\node (1-2v6) [v:mainempty,position=180:11.5mm from L1C2] {};
				\node (1-2v7) [v:main,position=216:11.5mm from L1C2] {};
				\node (1-2v8) [v:mainempty,position=252:11.5mm from L1C2] {};
				\node (1-2v9) [v:main,position=288:11.5mm from L1C2] {};
				\node (1-2v10) [v:mainempty,position=324:11.5mm from L1C2] {};
				
				\node (L1C3) [v:ghost,position=0:34.5mm from L1C2] {};
				
				\node (L1C3label) [v:ghost,position=270:19mm from L1C3] {$\mathscr{M}_{14}$};
				
				\node (1-3v1) [v:main,position=0:13mm from L1C3] {};
				\node (1-3v2) [v:mainempty,position=25.7143:13mm from L1C3] {};
				\node (1-3v3) [v:main,position=51.4286:13mm from L1C3] {};
				\node (1-3v4) [v:mainempty,position=77.1429:13mm from L1C3] {};
				\node (1-3v5) [v:main,position=102.857:13mm from L1C3] {};
				\node (1-3v6) [v:mainempty,position=128.571:13mm from L1C3] {};
				\node (1-3v7) [v:main,position=154.286:13mm from L1C3] {};
				\node (1-3v8) [v:mainempty,position=180:13mm from L1C3] {};
				\node (1-3v9) [v:main,position=205.714:13mm from L1C3] {};
				\node (1-3v10) [v:mainempty,position=231.429:13mm from L1C3] {};
				\node (1-3v11) [v:main,position=257.143:13mm from L1C3] {};
				\node (1-3v12) [v:mainempty,position=282.857:13mm from L1C3] {};
				\node (1-3v13) [v:main,position=308.571:13mm from L1C3] {};
				\node (1-3v14) [v:mainempty,position=334.286:13mm from L1C3] {};
				
				\node (L1C4) [v:ghost,position=0:20mm from L1C3] {$\cdots$};
				
			\end{pgfonlayer}
			
			
			\begin{pgfonlayer}{background}
				
				\draw [e:main,bend right=20] (1-1v1) to (1-1v2);
				\draw [e:main,bend right=20] (1-1v2) to (1-1v3);
				\draw [e:main,bend right=20] (1-1v3) to (1-1v4);
				\draw [e:main,bend right=20] (1-1v4) to (1-1v5);
				\draw [e:main,bend right=20] (1-1v5) to (1-1v6);
				\draw [e:main,bend right=20] (1-1v6) to (1-1v1);
				
				\draw [e:main,bend right=12] (1-2v1) to (1-2v2);
				\draw [e:main,bend right=12] (1-2v2) to (1-2v3);
				\draw [e:main,bend right=12] (1-2v3) to (1-2v4);
				\draw [e:main,bend right=12] (1-2v4) to (1-2v5);
				\draw [e:main,bend right=12] (1-2v5) to (1-2v6);
				\draw [e:main,bend right=12] (1-2v6) to (1-2v7);
				\draw [e:main,bend right=12] (1-2v7) to (1-2v8);
				\draw [e:main,bend right=12] (1-2v8) to (1-2v9);
				\draw [e:main,bend right=12] (1-2v9) to (1-2v10);
				\draw [e:main,bend right=12] (1-2v10) to (1-2v1);
				
				\draw [e:main,bend right=7] (1-3v1) to (1-3v2);
				\draw [e:main,bend right=7] (1-3v2) to (1-3v3);
				\draw [e:main,bend right=7] (1-3v3) to (1-3v4);
				\draw [e:main,bend right=7] (1-3v4) to (1-3v5);
				\draw [e:main,bend right=7] (1-3v5) to (1-3v6);
				\draw [e:main,bend right=7] (1-3v6) to (1-3v7);
				\draw [e:main,bend right=7] (1-3v7) to (1-3v8);
				\draw [e:main,bend right=7] (1-3v8) to (1-3v9);
				\draw [e:main,bend right=7] (1-3v9) to (1-3v10);
				\draw [e:main,bend right=7] (1-3v10) to (1-3v11);
				\draw [e:main,bend right=7] (1-3v11) to (1-3v12);
				\draw [e:main,bend right=7] (1-3v12) to (1-3v13);
				\draw [e:main,bend right=7] (1-3v13) to (1-3v14);
				\draw [e:main,bend right=7] (1-3v14) to (1-3v1);
				
				\draw [e:main] (1-1v1) to (1-1v4);
				\draw [e:main] (1-1v2) to (1-1v5);
				\draw [e:main] (1-1v3) to (1-1v6);
				
				\draw [e:main] (1-2v1) to (1-2v6);
				\draw [e:main] (1-2v2) to (1-2v7);
				\draw [e:main] (1-2v3) to (1-2v8);
				\draw [e:main] (1-2v4) to (1-2v9);
				\draw [e:main] (1-2v5) to (1-2v10);
				
				\draw [e:main] (1-3v1) to (1-3v8);
				\draw [e:main] (1-3v2) to (1-3v9);
				\draw [e:main] (1-3v3) to (1-3v10);
				\draw [e:main] (1-3v4) to (1-3v11);
				\draw [e:main] (1-3v5) to (1-3v12);
				\draw [e:main] (1-3v6) to (1-3v13);
				\draw [e:main] (1-3v7) to (1-3v14);
				
			\end{pgfonlayer}	
		\end{tikzpicture}
	\end{center}
	\caption{The odd M\"obius ladders. They build a chain with respect to the matching minor relation due to \cref{lemma:moebiuschain}.}
	\label{fig:macuaigbraces}
\end{figure}

The odd M\"obius ladders builds a chain with respect to the matching minor relation.

\begin{lemma}\label{lemma:moebiuschain}
Let $k\in\N$ be a positive integer.
Then $\mathscr{M}_{4k+2}$ is a matching minor of $\mathscr{M}_{4k+6}$.
\end{lemma}

\begin{proof}
Consider $\mathscr{M}_{4k+6}$ and the edge $v_0^1v_{k+2}^2$.
In $\mathscr{M}_{4k+6}-v_0^1v_{k+2}^2$ the vertices $v_0^1$ and $v_{k+2}^2$ are the only two vertices of degree two and thus we may bicontract both of them.
As $v_0^1$ and $v_{k+2}^2$ come from different colour classes, their neighbourhoods are disjoint and so the resulting graph, let us call it $B$, has $4k+2$ vertices.
Moreover, let us denote by $u^1$ the vertex obtained by identifying $v_{2k+2}^2$, $v_0^1$, and $v_0^2$ and by $u^2$ the vertex obtained from the identification of $v_{k+2}^1$, $v_{k+2}^2$, and $v_{k+3}^1$.
Observe that
\begin{align*}
	\Brace{u^1,v_1^1,v_1^2,\dots,v_{k+1}^2,u^2,v_{k+3}^2,\dots,v_{2k+2}^1,u^1}
\end{align*}
is still a Hamilton cycle of $B$.
Indeed, for each $i\in{0,\dots,2k+2}\setminus\Set{0,k+2,k+3}$ we still have the edge $v_i^1v_{k+i\Brace{\mod 2k+3}}$ in $B$.
Additionally we obtain the edge $u^1u^2$.
By renumbering it becomes apparent that $B$ is indeed isomorphic to $\mathscr{M}_{4k+2}$.
\end{proof}

We make use of the following two results by McCuaig.

\begin{lemma}[{\cite[Lemma 57]{mccuaig2004polya}}]
    \label{lem:K33_implies_moebius_ladder}
    Let $B$ be a bipartite graph with a perfect matching $M$ and a bisubdivision of $K_{3,3}$ as conformal subgraph.
    Then $G$ has an $\mathscr{M}_{4n+2}$ bisubdivision as $M$-conformal subgraph for some $n \geq 1.$
\end{lemma}

\begin{theorem}[{\cite[Theorem 29]{mccuaig2001brace}}]
    \label{thm:cube_or_K33}
    Every brace except $C_4$ has a bisubdivision of $K_{3,3}$ or the cube as conformal subgraph.
\end{theorem}

Together they imply the following statement.

\begin{corollary}
    \label{cor:M10}
    Let $B$ be bipartite matching covered graph.
    The graph $B$ does not contain the cube or the M\"obius ladder $\mathscr{M}_{10}$ as a matching minor if and only if all braces of $B$ are isomorphic to $C_4$ or $K_{3,3}.$
\end{corollary}
\begin{proof}
    Assume $B$ has a brace that is not isomorphic to $C_4$ or $K_{3,3}.$
    Then, by \cref{thm:cube_or_K33} $B$ contains a bisubdivision of $K_{3,3}$ or the cube as conformal subgraph.
    If $B$ contains the cube as conformal subgraph, we are done, thus assume $B$ contains $K_{3,3}$ as conformal subgraph.
    Then, by \cref{lem:K33_implies_moebius_ladder}, $B$ contains an $\mathscr{M}_{4n+2}$ bisubdivision as conformal subgraph for some $n \geq 1.$
    By, \cref{lemma:moebiuschain} this implies that $B$ contains $\mathscr{M}_{10}$ as matching minor.

    For the reverse direction assume that $B$ contains the cube or $\mathscr{M}_{10}$ as a matching minor.
    As they are braces, there is a brace $J$ of $B$ containing the cube or $\mathscr{M}_{10}$ as a matching minor.
    Thus, $J$ has to be of size at least 8 and cannot be isomorphic to $C_4$ or $K_{3,3}.$
\end{proof}

\Cref{thm:mwidth2graphs,cor:M10} now immediatly yield \cref{thm:Mpmw2}.

\cubeormoebiustheorem*

\bibliographystyle{alpha}
\bibliography{literature}

\end{document}